\documentclass[11pt]{article}

\usepackage{amstext}
\usepackage{amsmath}
\usepackage{amssymb}
\usepackage{amsthm}
\usepackage{a4wide}
\usepackage{enumerate}
\usepackage{mathrsfs}

\usepackage{pictexwd,dcpic}

\newtheorem{theo}{Theorem}[section]
\newtheorem{lemma}[theo]{Lemma}
\newtheorem{cor}[theo]{Corollary}
\newtheorem{prop}[theo]{Proposition}

\theoremstyle{definition}
\newtheorem{definition}[theo]{Definition}
\newtheorem{remark}[theo]{Remark}

\newcommand{\N}{\mathbb{N}}

\newcommand{\R}{\mathbb{R}}

\newcommand{\eps}{\varepsilon}

\newcommand{\weakly}{\rightharpoonup}
\newcommand{\weakstar}{\stackrel{\ast}{\rightharpoonup}}

\newcommand{\gre}[1]{\stackrel{#1}{\rightharpoonup}}

\def\hom{\mathop{\rm hom}}

\def\qcls{\mathop{\rm qcls}}
\def\qc{\mathop{\rm qc}}
\def\sm{\mathop{\rm sym}}

\def\lsc{\mathop{\rm lsc}}

\def\dev{\mathop{\rm dev}}
\def\tr {\mathop{\rm tr}}
\def\DIV{\mathop{\rm div}}
\def\im{\mathop{\rm im}}
\def\Int{\mathop{\rm Int}}
\def\dom{\mathop{\rm dom}}
\def\Int{\mathop{\rm Int}}

\def\E {\mathfrak{E}}
\def\F {{\cal F}}
\def\L {\mathcal L}
\def\G {{\cal G}}

\def\i {\infty}
\def\Y {{\mathbb I}}

\def\x {\ dx}
\def\y {\ dy}

\def\sm{\mathop{\rm sym}}
\newcommand{\M}[2]{\R^{#1 \times #2}}

\def\prostor{\, \underline{\phantom{x}} \,}

\def\XXint#1#2#3{{\setbox0=\hbox{$#1{#2#3}{\int}$}
     \vcenter{\hbox{$#2#3$}}\kern-.5\wd0}}

\usepackage{color}

\begin{document}

\begin{center}
\begin{Large}
Homogenization and the limit of vanishing hardening in Hencky plasticity with non-convex potentials 
\end{Large}
\end{center}

\begin{center}
\begin{large}
Martin Jesenko and Bernd Schmidt\\[0.2cm]
\end{large}
\begin{small}
Institut f{\"u}r Mathematik,\\ 
Universit{\"a}t Augsburg\\ 
86135 Augsburg, Germany\\ 
{\tt martin.jesenko@math.uni-augsburg.de}\\
{\tt bernd.schmidt@math.uni-augsburg.de}
\end{small}
\end{center}

\begin{abstract}
We prove a homogenization result for Hencky plasticity functionals with non-convex potentials. We also investigate the influence of a small hardening parameter and show that homogenization and taking the vanishing hardening limit commute. 
\end{abstract}
Mathematics Subject Classification 49J45 · 74C05 · 74Q05 

\tableofcontents

\section{Introduction and main results}
%
\begin{trivlist}
\item
A classical problem in material science is to deduce effective properties of composites with highly oscillatory material properties through a homogenization procedure. For hyper-elastic solids with a fine periodic structure of microscopic size $\eps \ll 1$, an effective model can be obtained rigorously with variational methods in the limit $\eps \to 0$ of the corresponding family of $\eps$-dependent stored energy functionals. The resulting $\Gamma$-limit is an integral functional with a homogeneous stored energy function, in which the microstructural effects are homogenized in such a way that minimizers (under suitable boundary conditions and applied forces) approximate minimizers of the $\eps$-dependent functionals. 
\item
In linearized elasticity the stored energy function, which acts on the linearized strain, is a `single-well' quadratic function minimized at $0$. The homogenization problem for such functions was being extensively studied already in the seventies. We refer to \cite{Marcellini:78} and the references therein. More generally, in \cite{Marcellini:78} Marcellini considers convex integrands with a suitable $p$-growth assumption. In finite elasticity this problem has been solved by Braides \cite{Braides:85} and M{\"u}ller \cite{Mueller:87}. In this setting the problem is considerably more involved as physically reasonable energy functions are necessarily non-convex: a typical single-well energy function is minimized on the non-convex set ${\rm SO}(3)$. 
\item
Yet, even for small strains non-convex energy densities are encountered when modeling materials with different `variants' represented by different energy wells. Such multi-well potentials occur, e.g., in the martensitic phase of shape memory alloys (see, e.g., \cite{Bhattacharya} for more details). The energy functionals are then only `geometrically linear' in the sense that the energy density, while still only dependent on the linearized strain, may be a general non-quadratic function, see, e.g., \cite{Khachaturyan:67,Khachaturyan:83,KhachaturyanShatalov:69,Roitburd:67,Roitburd:78,Schmidt}. 
\item
In perfect small strain elastoplasticity (i.e.\ with zero hardening) the material behavior beyond the elastic regime is modeled by a flat relation between stress and deviatoric strain. After passing the yield surface, the stored energy thus grows linearly in the deviatoric strain, yet still quadratic in the hydrostatic strain. We refer to the seminal article \cite{Suquet} of Suquet for a mathematical treatment of the evolution problem in perfect elastoplasticity as well as to the classical volumes of Duvaut and Lions \cite{DuvautLions} and of Hl\'{a}va\v{c}ek and Ne\v{c}as \cite{HlavacekNecas} and the more recent book by Han and Reddy \cite{HanReddy} for an introduction to the mathematical description of plastic behavior and a discussion of various different models. 
\item
A static description of this `pseudoelastic' regime through stored energy functionals with mixed linear-quadratic growth is referred to as the Hencky plasticity model. (Of course, its validity is restricted to one-time loading since it cannot include hysteresis effects.) Due to the linear growth conditions, the Hencky plasticity functional is not coercive on Sobolev spaces, and more general displacements in the space $BD$ of functions of bounded deformation have to be taken into account. While existence results in this setting have been obtained by Anzellotti and Giaquinta in \cite{AnzellottiGiaquinta:80,AnzellottiGiaquinta}, only recently, Mora showed that such convex $BD$ functionals are in fact the relaxed Hencky plasticity functionals on Sobolev spaces. We also refer to the monograph of Temam \cite{Temam} for related results. 
\item
The situation, however, is considerably easier if one introduces a finite hardening parameter, which leads to regularized functionals on $W^{1,2}$. The natural question, if Hencky plasticity is an effective model in the limit of vanishing hardening has been positively answered (even in a time-dependent) setting in \cite{BartelsMielkeRoubicek}. 
\item
A homogenization result in convex Hencky plasticity was obtained by Demengel and Qi, \cite{DemengelQi}. The resulting functional then acts on general $BD$-functions, and the limiting stored energy is given in terms of a convex function applied to the limiting strain in the sense of \cite{DemengelTemam:84,DemengelTemam}, which now is merely a measure rather than a function. 
\end{trivlist}
\subsection*{Main results} 
\begin{trivlist}
\item
A main aim of the present contribution is to derive a homogenized model in small strain Hencky plasticity with non-convex potentials, see Theorem \ref{theo:Hencky-epilog}. As alluded to above, such functionals may describe fine mixtures of martensites in shape memory alloys beyond their (super-)elastic regime. More precisely, we will consider functionals with linear growth in the deviatoric part and quadratic growth in the trace, which merely satisfy an asymptotic convexity condition at infinity. In particular, we allow for general non-linear and non-convex stored energy functions within the (super-)elastic regime. 
\item
Let $\Omega \subset \R^n$ be a domain with Lipschitz boundary. For a function of bounded deformation $u \in BD(\Omega)$, we denote by $(\E u) \, \L^{n} $ the absolutely continuous part and by $E^s u$ the singular part with respect to Lebesgue measure  $\L^{n}$ of the symmetrized distributional derivative $Eu = \frac{1}{2}(Du + (Du)^T)$. (See Section~\ref{section:Preliminaries} for the definition and basic properties of $BD(\Omega)$.) If $E^s u = 0$, we also write $u \in LD(\Omega)$. Adapted to the Hencky plasticity setting, we introduce 
\begin{align*}
  U( \Omega ) 
  &:= \{ u \in BD( \Omega ) : \DIV u \in L^{2}( \Omega ) \}, \\ 
  LU( \Omega ) 
  &:= \{ u \in LD( \Omega ) : \DIV u \in L^{2}( \Omega ) \}.
\end{align*}
(Basic properties of these spaces are discussed in Subsection~\ref{subsection:U}.) Suppose that $ f : \R^{n} \times \M{n}{n}_{\sm} \to \R $ is a Carath\'{e}odory function that is $ \Y^{n} $-periodic in the first variable (where $\Y = (0,1)$) and satisfies the growth condition of Hencky plasticity 
\begin{align}\label{eq:Hencky-growth} 
  \alpha( | X_{\dev} | + ( \tr X )^{2} ) \le f( x , X ) \le \beta( | X_{\dev} | + ( \tr X )^{2} + 1 ) 
\end{align}
for suitable $ \alpha , \beta > 0 $ for a.e.~$ x \in \Omega $ and every $ X \in \M{n}{n}_{\sm} $. Here $X_{\dev} = X - \frac{\tr X}{n} I$ denotes the deviatoric part of $X$. 
We will also write $ \M{n}{n}_{\dev} = \{ X \in \M{n}{n}_{\sm} : \tr X = 0 \} $. In analogy to the elastic setting we define the homogenized density by 
\begin{align}\label{eq:homogenization-formula} 
 f_{\hom}(X) 
:= \inf_{k \in \N} \inf_{ \varphi \in C_{c}^{\i}( k \Y^{n} ; \R^{n} ) } \frac{1}{ k^{n} } \int_{ k \Y^{n} } f( x , X + \E \varphi(x) ) \x.
\end{align}
We also introduce the {\em asymptotic function} $g^{\#}$ of a general (not necessarily convex) function $g : \R^N \to \R$ by setting 
\[ g^{\#}(X) := \limsup_{ t \to \i } \frac{ g(tX) }{t}. \]
We finally require $f$ to satisfy the following {\it asymptotic convexity condition}: 
For every $ \eta > 0 $ there are $ \beta_{\eta} > 0 $ and a Carath\'{e}odory function $ c^{\eta} : \R^{n} \times \M{n}{n}_{\sm} \to \R $ 
that is $ \Y^{n} $-periodic in the first variable and convex in the second such that
\begin{align}\label{eq:asymptotic-convexity} 
  | f(x,X) - c^{\eta}(x,X) | \le \eta (|X_{\dev}| + (\tr X)^2) + \beta_{\eta}. 
\end{align}
for a.e.~$ x \in \R^{n} $ and all $ X \in \M{n}{n}_{\sm} $. 
\begin{theo}
\label{theo:Hencky-epilog}
Suppose $ f : \R^{n} \times \M{n}{n}_{\sm} \to \R $ is a Carath\'{e}odory function 
that is $ \Y^{n} $-periodic in the first variable and satisfies \eqref{eq:Hencky-growth} and \eqref{eq:asymptotic-convexity}. 
Then the functionals 
\[ \F_{\eps}(u) := \left\{ \begin{array}{ll}
\int_{ \Omega } f \big( \tfrac{x}{ \eps } , \E u(x) \big) \x, & u \in LU( \Omega ), \\
\i, & {\rm else,}
\end{array} \right. \]
$\Gamma( L^{1} )$-converge to 
\[ \F_{\hom}(u) := \left\{ \begin{array}{ll}
\int_{ \Omega } f_{\hom} \big( \E u(x) \big) \x
+ \int_{ \Omega } ( f_{\hom} )^{\#} \big( \tfrac{ d E^{s} u }{ d | E^{s} u | }(x) \big) \ d| E^{s} u |(x), & u \in U( \Omega ), \\
\i, & {\rm else.}
\end{array} \right. \]
\end{theo}
\begin{remark}
\label{remark:Hencky-epilog}
Without condition \eqref{eq:asymptotic-convexity} we still have 
\[ \Gamma( L^{1} ) \mbox{-} \limsup_{ \eps \to 0 } \F_{\eps} \le \F_{\hom}, \]
on $L^1(\Omega; \R^n)$ and 
\[ \Gamma( L^{1} ) \mbox{-} \lim_{ \eps \to 0 } \F_{\eps}(u) = \F_{\hom}(u) \]
for every $ u \in LU( \Omega ) $, i.e., only the $\liminf$-inequality at singular points requires \eqref{eq:asymptotic-convexity}. 
\end{remark} 
\item
By adding a small hardening parameter $\delta$, one obtains regularized functionals on Sobolev spaces to which the homogenization results in \cite{Braides:85,Mueller:87} apply. However, it is by no means obvious if upon sending $\delta$ to zero, the regularized homogenized functionals will converge to the homogenized Hencky plasticity functional. Our second main result shows that this is indeed the case: Homogenization and the taking the limit of vanishing hardening commute. 
\item
For $f$ as above and any $ \delta \ge 0 $ we set
\[ f^{(\delta)} : \R^{n} \times \M{n}{n}_{\sm} \to \R,
\quad
f^{(\delta)}( x , X ) := f( x , X ) + \delta | X_{\dev} |^{2}, \]
so that $f^{(\delta)}$ satisfies a standard $2$-growth assumption. We define the homogenized density $f_{\hom}^{(\delta)}$ in analogy to \eqref{eq:homogenization-formula}. 
\begin{theo}
\label{theo:Hencky-commutability}
Suppose $ f : \R^{n} \times \M{n}{n}_{\sm} \to \R $ is a Carath\'{e}odory function 
that is $ \Y^{n} $-periodic in the first variable and satisfies \eqref{eq:Hencky-growth} and \eqref{eq:asymptotic-convexity}. Let 
\[ \F^{(\delta)}_{\eps}(u) := \left\{ \begin{array}{ll}
\int_{ \Omega } f^{(\delta)}( \tfrac{x}{ \eps } , \E u(x) ) \x, & u \in W^{1,2}( \Omega ; \R^{n} ), \\
\i, & {\rm else}.
\end{array} \right. \]
and 
\[ \F^{(\delta)}_{\hom}(u) := \left\{ \begin{array}{ll}
\int_{ \Omega } f^{(\delta)}_{ \hom }( \E u(x) ) \x, & u \in W^{1,2}( \Omega ; \R^{n} ), \\
\i, & {\rm else},
\end{array} \right. \]
Then the following diagrams, in which horizontally $\delta \to 0$ and vertically $\eps \to 0$, commute:
\[
\begindc{\commdiag}[500]
\obj(0,2)[aa]{$ \F^{ (\delta) }_{\eps} $}
\obj(2,2)[bb]{$ \F^{ (0) }_{\eps} $}
\obj(0,0)[cc]{$ \F^{ (\delta) }_{\hom} $}
\obj(2,0)[dd]{$ \F_{\hom} $}
\mor{aa}{bb}{ptw. falling}
\mor{aa}{cc}{$ \Gamma $}
\mor{bb}{dd}{$ \Gamma $}
\mor{cc}{dd}{$ \Gamma $}
\enddc
\qquad \mbox{and} \qquad
\begindc{\commdiag}[500]
\obj(0,2)[aa]{$ \F^{ (\delta) }_{\eps} $}
\obj(2,2)[bb]{$ \lsc \F^{ (0) }_{\eps} = \lsc \F_{\eps}$}
\obj(0,0)[cc]{$ \F^{ (\delta) }_{\hom} $}
\obj(2,0)[dd]{$ \F_{\hom} $}
\mor{aa}{bb}{$ \Gamma $}
\mor{aa}{cc}{$ \Gamma $}
\mor{bb}{dd}{$ \Gamma $}
\mor{cc}{dd}{$ \Gamma $}
\enddc
\]
(All $ \Gamma $-limits are with respect to the $ L^{1} $-norm.) Here $\lsc$ denotes the $L^1$-lower semicontinuous envelope. 
\end{theo}
\item
This result shows that in fact the influence of highly oscillating material parameters and small hardening decouple. In this sense, Theorem \ref{theo:Hencky-commutability} provides an extension of our general commutability result \cite{JesenkoSchmidt} to a -- yet specific -- situation with mixed growth conditions. 
\item
We remark that the treatment of non-convex functionals requires a different approach as compared to \cite{DemengelQi}, where extensive use is made of the fact that a convex energy function can directly be applied to the measure representing the symmetrized derivative of a $BD$ function. Instead, our strategy to prove the $\limsup$-inequality in Theorem \ref{theo:Hencky-epilog} rests upon the density of smooth functions and continuity of $\F_{\hom}$ in a suitable intermediate topology on $U(\Omega)$. To this end, we borrow ideas from Kristensen and Rindler \cite{KristensenRindler-Relax} for functionals with purely linear growth and prove the following continuity result which might be of some independent interest. (See Definitions~\ref{defi:strict-conv} and \ref{defi:sym-rk-one} for the notions of $ \langle \cdot \rangle $-strict continuity and symmetric-rank-one-convexity.) 
\begin{prop}
\label{prop:c-strict-continuity}
Let $ f : \Omega \times \M{n}{n}_{\sm} \to \R $ be a continuous function 
that is symmetric-rank-one-convex in the second variable
and that satisfies the Hencky plasticity growth condition \eqref{eq:Hencky-growth}. Denote $ f_{\dev} := f|_{ \Omega \times \M{n}{n}_{\dev} } $.
Suppose that
\[ ( f_{\dev} )^{\i}( x_{0} , P_{0} ) = \limsup_{ P \to P_{0} , t \to \i } \frac{ f_{\dev}( x_{0} , t P ) }{t} \]
is for every fixed $ P_{0} \in \M{n}{n}_{\dev} $ a continuous function of $ x_{0} $. Then the functional
\[ \F(u) = \int_{\Omega} f \big( x , \E u(x) \big) \x + \int_{\Omega} ( f_{\dev} )^{\i} \big( x , \tfrac{ d E^{s} u }{ d | E^{s} u |}(x) \big) \ d | E^{s} u |(x) \]
is $ \langle \cdot \rangle $-strictly continuous on $ U( \Omega ) $.
\end{prop}
\item 
The $\liminf$-inequality at regular points is obtained by a localization and slicing method. However, due to the non-locality of $U(\Omega)$, additional terms have to be introduced with the help of the Bogovskii-operator in order to achieve quadratic integrability of the divergence. In analyzing the singular points (as well as in the $\limsup$-inequality), we use a novel rank-1-theorem for $BD$-functions, only recently proved by De Philippis und Rindler in \cite{DePhilippisRindler}. While we conjecture our result to hold true in more generality, at this point we make use of the aforementioned asymptotic convexity. At regular points we use that $BD$-functions are $L^q$-differentiable a.e.\ for some $q > 1$. While $L^1$-differentiability had been established already in \cite{AmbrosioCosciaDalMaso}, for $ 1 < q < \frac{n}{n-1} $ this was only recently obtained in \cite{ABC}. We include an alternative proof which also covers the case $ q = \frac{n}{n-1} $, thus answering positively a question raised in \cite{ABC}, see Corollary \ref{cor:BD-approx-diff-L1*}. 
\item
The paper is organized as follows. In Section \ref{section:Preliminaries} we collect some basic material on the function spaces $BD$, $LD$, $U$ and $LU$. We also prove the higher order approximate differentiability for functions of bounded deformation and a Helmholtz decomposition result on $U$. The following Section~\ref{section:Hencky-commutability} contains the proofs of our main results. We begin in Section~\ref{subsection:Hencky-setting} by establishing the straightforward limits $\delta \to 0$ while $\eps > 0$ and $\eps \to 0$ while $\delta > 0$ in Theorem~\ref{theo:Hencky-commutability} and also analyze basic properties of the homogenized density $f_{\hom}$. Section~\ref{section:Hencky-limsup} then gives the proof of Proposition~\ref{prop:c-strict-continuity}, from which also the proof of the $\limsup$-inequality in Theorem~\ref{theo:Hencky-epilog} is readily deduced. In Section~\ref{section:Hencky-liminf} we then show that also the $\liminf$-inequality is satisfied and thus complete the proofs of Theorems~\ref{theo:Hencky-epilog} and \ref{theo:Hencky-commutability}. In the last Section~\ref{section:Hencky-relax-KK} we discuss a complemetary relaxation result in the location independent case. 
\end{trivlist}
%
\section{Preliminaries}\label{section:Preliminaries}

For convenience of the reader, we first briefly review the definition and basic properties of functions of bounded deformation. 
We then establish a higher-order approximate differentiability result for these functions. 
Finally we discuss the spaces $LU$ and $U$ in some detail including a Helmholtz decomposition result on $U$.  

\subsection{Functions of bounded deformation}

\begin{trivlist}
\item
We may define $ \E u $ via distribution or directly as the only (if existing) function that fulfils
\[ \int_{\Omega} \E u(x) \cdot \Phi(x) \x 
= - \int_{\Omega} u(x) \cdot \DIV( \Phi(x)_{\sm} ) \x \]
for all $ \Phi \in C_{c}^{\i}( \Omega ; \M{n}{n} ) $. We will denote the corresponding space by
\[ LD( \Omega ) := \{ u \in L^{1}( \Omega ; \R^{n} ) : \E u \in L^{1}( \Omega ; \M{n}{n} ) \}. \]
It becomes a Banach space when equipped with the natural norm
\[ \| u \|_{ LD( \Omega ) } := \| u \|_{ L^{1}( \Omega ; \R^{n} ) } + \| \E u \|_{ L^{1}( \Omega ; \M{n}{n} ) }. \]
For the properties of this space, we refer to, e.g., \cite[Section II.1]{Temam}. 
Analogously to the space of functions of bounded variation for the full-gradient case, 
one introduces the space of functions of bounded deformation as follows:
If the mapping
\[ C_{c}^{\i}( \Omega ; \M{n}{n} ) \to \R,
\quad
\Phi \mapsto - \int_{\Omega} u(x) \cdot \DIV( \Phi(x)_{\sm} ) \x, \]
may be extended to a bounded linear functional on $ C_{0}( \Omega ; \M{n}{n} ) $, i.e.~to a Radon measure,
then we denote this functional by $ E u \in M( \Omega ; \M{n}{n} ) $. 
(Clearly, it must lie in $ M( \Omega ; \M{n}{n}_{\sm} ) $.)
The space of functions of bounded deformation is defined by
\[ BD( \Omega ) := \{ u \in L^{1}( \Omega ; \R^{n} ) : E u \in M( \Omega ; \M{n}{n} ) \}. \]
Equipped with the norm
\[ \| u \|_{ BD( \Omega ) } := \| u \|_{ L^{1}( \Omega ; \R^{n} ) } + \| E u \|_{ M( \Omega ; \M{n}{n} ) }, \]
it is also a Banach space.
\item
As in the space of functions of bounded variation, we introduce (with abuse of terminology) the weak convergence:
A sequence $ \{ u_{j} \}_{ j \in \N } $ {\it converges weakly in $ BD( \Omega ) $} to $ u $,
denoted $ u_{j} \weakly u $,
if
\[ u_{j} \to u \quad \mbox{in } L^{1}( \Omega ; \R^{n} )
\quad \mbox{and} \quad
E u_{j} \weakstar E u \quad \mbox{in } M( \Omega ; \M{n}{n} ). \]
Let us mention that every bounded sequence in $ BD( \Omega ) $ contains a weakly convergent subsequence.

\item
If additionally $ | E u_{j} |( \Omega ) \to | E u |( \Omega ) $, then we speak of {\it strict} or {\it intermediate convergence}.
This topology is actually induced by the metric
\[ d(u,v) := \| u - v \|_{ L^{1} } + \big| | Eu |( \Omega ) - | Ev |( \Omega ) \big|. \]
\item
Finally, let $ c : \M{n}{n}_{\sm} \to [ 0 , \i ) $ be a convex function with linear upper bound.
A sequence $ \{ u_{j} \}_{ j \in \N } $ {\it converges $ c $-strictly in $ BD( \Omega ) $} to $ u $,
symbolically $ u_{j} \gre{c} u $, if 
\begin{itemize}
\item
$ u_{j} \to u $ in $ L^{1}( \Omega ; \R^{n} ) $,
\item
$ | E u_{j} |( \Omega ) \to | E u |( \Omega ) $,
\item
$ \int_{ \Omega } c( E u_{j} ) \to \int_{ \Omega } c( E u ) $.
\end{itemize}
For a general discussion about a convex linearly bounded function of a measure
including the density results for $c$-strict topology, see \cite[Section~II.4]{Temam} or \cite{DemengelTemam:84}.
Let us just mention that every $ \mu \in M( \Omega ; \M{n}{n}_{\sm} ) $ has the Lebesgue decomposition 
\[ \mu = \mu^{a} + \mu^{s} \]
into the absolutely continuous and the singular part with respect to the Lebesgue measure $ \L^{n} $.
Then
\[ c( \mu ) := c \Big( \frac{ d \mu^{a} }{ d \L^{n} } \Big) \L^{n} + c^{\i} \Big( \frac{ d \mu^{s} }{ d | \mu^{s} | } \Big) | \mu^{s} |. \]
Moreover, the denotation $ \int_{A} c( \mu ) $ used above is merely another way for writing $ c( \mu )(A) $.
\item
For $ u \in BD( \Omega ) $ we decompose $ Eu $ with respect to the Lebesgue measure $ \L^{n} $ into
\[ Eu =: \E u \ \L^{n} + E^{s} u. \]
By denoting the density of the absolutely continuous part by $ \E u $, 
we have just extended the definition from $ LD( \Omega ) $. Clearly,
\[ LD( \Omega ) = \{ u \in BD( \Omega ) : E^{s} u = 0 \}. \]
\item
Although being larger than the space of functions of bounded variation, 
some properties still hold also for the space of functions of bounded deformation.
E.g., it is possible to define boundary values, 
as was shown in \cite[Theorem II.2.1]{Temam} for domains with $ C^{1} $-boundary 
and extended recently by Babadjian, see \cite{Babadjian}: 
\begin{theo}
Let $ \Omega \subset \R^{n} $ be a bounded Lipschitz domain. 
There exists a unique linear continuous mapping $ \gamma : BD( \Omega ) \to L^{1}( \partial \Omega ; \R^{n} ) $
such that the following integration by parts formula holds: 
for every $ u \in BD( \Omega ) $ and $ \varphi \in C^{1}( \R^{n} ) $
\[ \int_{ \Omega } u(x) \odot \nabla \varphi(x) \x
+ \int_{ \Omega } \varphi(x) \ d Eu(x)
= \int_{ \partial \Omega } \varphi(x) \ \gamma(u)(x) \odot \nu(x) \ d {\mathscr H}^{n-1}(x). \]
For all $ u \in BD( \Omega ) \cap C( \overline{ \Omega } ; \R^{n} ) $, it holds
$ \gamma(u) = u|_{ \partial \Omega }  $.
In point of fact, $ \gamma $ is even continuous if $ BD( \Omega ) $ is endowed with the strict topology.
\end{theo}
\item 
Here $ \nu $ denotes the outer unit normal vector to $ \partial \Omega $ 
and $ {\mathscr H}^{n-1} $ the $ (n-1) $-dimensional Hausdorff measure. Henceforth, we will for $ a,b \in \R^{n} $ write
\[ a \odot b := \tfrac{1}{2} ( a \otimes b + b \otimes a ). \]
Theorems II.2.2 and II.2.4 in \cite{Temam} cover the subject of embeddings in $ L^{p} $-spaces, see also \cite{AnzellottiGiaquinta:80}:
\begin{theo}
\label{theo:BD-embeddings}
Let $ \Omega \subset \R^{n} $ be a bounded Lipschitz domain. 
There exists a natural bounded embedding
\[ BD( \Omega ) \hookrightarrow L^{ \frac{n}{n-1} }( \Omega ; \R^{n} ). \]
For every $ 1 \le q < \frac{n}{n-1} $ the embedding
\[ BD( \Omega ) \hookrightarrow L^{q}( \Omega ; \R^{n} ) \]
is even compact.
\end{theo}
\item
The singular part of the gradient of a function with bounded variation has a rank-one structure.
This was proved in \cite{Alberti} and is commonly known as Alberti's rank one theorem.
Very recently De Philippis and Rindler could show the analogous property for functions of bounded deformation
(see \cite[Theorem 1.7]{DePhilippisRindler}):
\begin{theo}
\label{theo:BD-Alberti}
Let $ \Omega \subset \R^{n} $ be an open set and let $ u \in BD( \Omega ) $.
Then, for $ | E^{s} u | $-a.e.~$ x \in \Omega $, there exist $ a(x) , b(x) \in \R^{n} \setminus \{0\} $ such that
\[ \frac{ d E^{s} u }{ d | E^{s} u | } = a(x) \odot b(x). \]
\end{theo}
\end{trivlist}
\subsection{Approximate differentiability}
\label{subsection:app-diff-BD}
\begin{trivlist}
\item
Although for the functions of bounded deformation the (full) gradient is in general not even a Radon measure, 
they still can be locally (on average) approximated by linear functions. More precisely,
\begin{theo}
\label{theo:BD-approx-diff-L1}
For every $ u \in BD( \Omega ) $ there exists a negligible set $ N \subset \Omega $ such that for all
$ x_{0} \in \Omega \setminus N $ there exists a matrix $ L_{ x_{0} } \in \M{n}{n} $ such that
\[ \lim_{ r \to 0 } \frac{1}{ r^{n} } \int_{ B_{r}( x_{0} ) }
\frac{ | u(x) - u( x_{0} ) - L_{ x_{0} } ( x - x_{0} ) | }{ r } \x = 0. \]
Therefore, $u$ is a.e.~approximately differentiable with $ L_{ x_{0} } = \nabla u( x_{0} ) $ being the approximate differential. 
Moreover, it holds $ \E u(x) = \sm \nabla u(x) $ for a.e.~$ x \in \Omega $.
The function $ \nabla u $ is in the weak-$ L^{1} $-space since
\[ | \{ x : | \nabla u(x) | > t \} | \le \frac{ c(n) }{ t } | Eu |( \Omega ). \]
\end{theo}
\item
For the proof, see \cite[Theorem 7.4]{AmbrosioCosciaDalMaso}. More general tools are presented in \cite{Hajlasz}.
We will improve this $L^{1}$-differentiability property 
to $ L^{ \frac{n}{n-1} } $-differentiability analogously as suggested in \cite{AmbrosioFuscoPallara} for functions of bounded variation. 
Let us mention that the proof of $ L^{q} $-differentiability for $ 1 \le q < \frac{n}{n-1} $ was recently done in \cite{ABC}.
\item
We will need an appropriate version of the Poincar\'{e}-Korn inequality 
(see \cite[Remark II.1.1]{Temam} and \cite[Corollary 4.20]{Bredies}):
\begin{theo}
Let $ {\cal R} $ be the set of all infinitesimal rigid motions in $ \R^{n} $. 
For a bounded Lipschitz domain $ \Omega \subset \R^{n} $,
there exists a constant $ C = C( \Omega ) $ such that for every $ u \in BD( \Omega ) $ 
\[ \min_{ \rho \in {\cal R} } \| u - \rho \|_{ L^{ \frac{n}{n-1} }( \Omega ; \R^{n} ) } \le C \| E u \|_{ M( \Omega ; \M{n}{n} ) }. \]
Therefore, for every projection $R$ onto $ {\cal R} $, we have (possibly for a larger constant)
\[ \| u - R(u) \|_{ L^{ \frac{n}{n-1} }( \Omega ; \R^{n} ) } \le C( R , \Omega ) \| E u \|_{ M( \Omega ; \M{n}{n} ) }. \]
\end{theo}
\item
We will need this inequality explicitly for balls so let us according to \cite[Remark~II.1.1]{Temam} define one projection. 
Fix $ B_{r}( x_{0} ) \subset \R^{n} $. Define
\[ R_{ x_{0} , r }(u)(x) := 
\frac{1}{ | B_{r} | } \int_{ B_{r}( x_{0} ) } u(y) \y
+ \frac{1}{ J_{r} } \left( \int_{ B_{r}( x_{0} ) } u(y) \times ( y - x_{0} ) \y \right) ( x - x_{0} ) \]
where $ J_{r} $ is defined by
\[ J_{r} 
:= \int_{ B_{r}(0) } y_{1}^{2} \y 
= \frac{ \pi^{n/2} }{ (n+2) \Gamma( \frac{n}{2} + 1 ) } \ r^{n+2}
=: J_{1} r^{n+2} \]
and for $ a , b \in \R^{n} $ we are denoting $a \times b := \tfrac{1}{2} ( a \otimes b - b \otimes a ) $. 
Using that every $ \rho \in {\cal R} $ can be written as a sum of a constant function and a linear combination of functions $ x \mapsto ( e_{i} \times e_{j} ) ( x - x_{0} ) $, 
one verifies that $ R_{ x_{0} , r } $ is indeed a projection.
\item
By changing variables in the Poincar\'{e} inequality, we see that the constant is translation and scaling invariant, 
i.e., $ C( R_{ x_{0} , r } , B_{r}( x_{0} ) ) = C( R_{0,1} , B_{1}(0) ) $.
\begin{cor}
\label{cor:BD-approx-diff-L1*}
Every $ u \in BD( \Omega ) $ is a.e.~$ L^{ \frac{n}{n-1} } $-differentiable, 
and for a.e.~$ x_{0} \in \Omega $ it holds
\[ \lim_{ r \to 0 } \frac{1}{ r^{n} } \int_{ B_{r}( x_{0} ) }
\left| \frac{ u(x) - u( x_{0} ) - \nabla u( x_{0} ) ( x - x_{0} ) }{ r } \right|^{ \frac{n}{n-1} } \x = 0. \]
\end{cor}
\begin{proof}
Let be $ x_{0} \in \Omega $ any point such that
\begin{itemize}
\item Theorem~\ref{theo:BD-approx-diff-L1} holds,
\item $ x_{0} $ is Lebesgue point of $ \E u $,
\item $ \lim_{ r \to 0 } \frac{1}{ r^{n} } | E^{s} u |(  B_{r}( x_{0} ) ) = 0 $.
\end{itemize} 
By the Besicovitch derivation theorem~\ref{theo:Besicovitch}, a.e.~$ x_{0} $ meets these conditions.
Applying the Poincar\'{e} inequality to the function $ \tilde{u}(x) := u(x) - u( x_{0} ) - L_{ x_{0} }( x - x_{0} ) $ yields
\[ \| \tilde{u} - R_{ x_{0} , r }( \tilde{u} ) \|_{ L^{ \frac{n}{n-1} } } \le C_{1} \| E \tilde{u} \|_{ M } \]
or
\begin{multline*} 
\left( \frac{1}{ r^{n} } \int_{ B_{r}( x_{0} ) }
\frac{ | \tilde{u}(x) - R_{ x_{0} , r }( \tilde{u} )(x) |^{ \frac{n}{n-1} } }{ r^{ \frac{n}{n-1} }} \x \right)^{ \frac{n-1}{n} }  \le \\
\le \frac{ C_{1} }{ r^{n} } \left( \int_{ B_{r}( x_{0} ) } | \E u(x) - \E u( x_{0} ) | \x + | E^{s} u |(  B_{r}( x_{0} ) ) \right). 
\end{multline*}
The right side converges for the chosen $ x_{0} $ to 0 as $ r \to 0 $.
Therefore, the claim will be proved when we show 
\[ \lim_{ r \to 0 } \frac{1}{ r^{n} } \int_{ B_{r}( x_{0} ) } 
\frac{ | R_{ x_{0} , r }( \tilde{u} )(x) |^{ \frac{n}{n-1}} }{ r^{ \frac{n}{n-1} }} \x = 0. \]
Denote $ \tilde{a}_{r} + \tilde{A}_{r} ( x - x_{0} ) := R_{ x_{0} , r }( \tilde{u} )(x) $.
According to Theorem~\ref{theo:BD-approx-diff-L1}
\[ \lim_{ r \to 0 } \frac{ \tilde{a}_{r} }{r}
= \lim_{ r \to 0 } \frac{1}{ | B_{r}( x_{0} ) | } \int_{ B_{r}( x_{0} ) } \frac{ \tilde{u}(x) }{r} \x
= 0. \]
Moreover, from
\[ \tilde{A}_{r}
= \frac{1}{ J_{1} r^{n+2} } \int_{ B_{r}( x_{0} ) } \tilde{u}(x) \times ( x - x_{0} ) \x, \]
it follows
\[ | \tilde{A}_{r} |
\le \frac{1}{ J_{1} r^{n+2} } \int_{ B_{r}( x_{0} ) } | \tilde{u}(x) | | x - x_{0} | \x
\le \frac{1}{ J_{1} r^{n} } \int_{ B_{r}( x_{0} ) } \frac{ | \tilde{u}(x) | }{r} \x \to 0. \]
Hence,
\[ \lim_{ r \to 0 } 
\frac{1}{ r^{n} } \int_{ B_{r}( x_{0} ) } \frac{ | \tilde{A}_{r}( x - x_{0} ) |^{ \frac{n}{n-1} } }{ r^{ \frac{n}{n-1} } } \x = 0. \qedhere \]
\end{proof}
\end{trivlist}

\subsection{The space $U$}
\label{subsection:U}
\begin{trivlist}
\item
For an integral functional whose density has Hencky plasticity growth,
its natural domain is the space
\[ LU( \Omega ) := \{ u \in LD( \Omega ) : \DIV u \in L^{2}( \Omega ) \}. \]
Endowed with 
\[ \| u \|_{ U( \Omega ) } := \| u \|_{ LD( \Omega ) } + \| \DIV u \|_{ L^{2}( \Omega ) }, \]
it is clearly a Banach space. If $ \Omega $ is a bounded Lipschitz domain, 
$ C^{\i}( \overline{ \Omega } ; \R^{n} ) $ is a dense subset 
(combine Proposition~I.1.3 with the proof of Theorem~II.3.4 in \cite{Temam}). 
We will analogously as in Sobolev spaces denote
\[ LU_{0}( \Omega ) := \overline{ C_{c}^{\i}( \Omega ; \R^{n} ) }. \]
Let $ f : \Omega \times \M{n}{n}_{\sm} \to \R $ be a Carath\'{e}odory function with Hencky plasticity growth.
For every $ X \in \M{n}{n}_{\sm} $, $ \varphi \in LU_{0}( \Omega ) $ and $ \eps > 0 $, 
there exists $ \varphi_{\eps} \in C^{\i}_{c}( \Omega ; \R^{n} ) $ 
such that
\[ \int_{ \Omega } f( x , X + \E \varphi(x) ) \x \ge \int_{ \Omega } f( x , X + \E \varphi_{\eps}(x) ) - \eps. \]
The proof follows the usual scheme of employing Fatou's lemma, 
passing to a.e.~pointwisely convergent sequence and using continuity of $f$ in the second variable.
\item
Due to the lack of weak compactness of bounded sequences in the space $ LU $, 
we introduce the corresponding space
\[ U( \Omega ) := \{ u \in BD( \Omega ) : \DIV u \in L^{2}( \Omega ) \} \]
(with the obvious norm). 
Clearly, since the trace part of $ E u $ is regular, we have $ E_{\dev}^{s} u = E^{s} u $.
The definitions and claims for $BD$ may be adapted in the following manner.
\begin{definition}\label{defi:strict-conv}
Let us have $ \{ u_{j} \}_{ j \in \N } \subset U( \Omega ) $ and $ u \in U( \Omega ) $.
\begin{enumerate}
\item
We say that $ \{ u_{j} \}_{ j \in \N } $ {\it weakly converges} to $ u $ in $ U( \Omega ) $, $ u_{j} \weakly u $, if
\begin{itemize}
\item
$ u_{j} \to u $ in $ L^{1}( \Omega ; \R^{n} ) $,
\item
$ E u_{j} \weakstar E u $ in $ M( \Omega ; \M{n}{n} ) $,
\item
$ \DIV u_{j} \weakly \DIV u $ in $ L^{2}( \Omega ) $. 
\end{itemize}
\item
The {\it strict} (or {\it intermediate}) {\it convergence}, $ u_{j} \gre{|.|} u $, means that 
\begin{itemize}
\item
$ u_{j} \to u $ in $ L^{1}( \Omega ; \R^{n} ) $,
\item
$ | E u_{j} |( \Omega ) \to | E u |( \Omega ) $,
\item
$ \DIV u_{j} \to \DIV u $ in $ L^{2}( \Omega ) $. 
\end{itemize}
(The underlying metric is clearly
\[ d(u,v) := \| u - v \|_{ L^{1} } + \big| | Eu |( \Omega ) - | Ev |( \Omega ) \big| + \| \DIV u - \DIV v \|_{ L^{2} }.) \]
\item
Let $ c : \M{n}{n}_{\sm} \to \R $ be a non-negative convex function with linear upper bound.
Then $ \{ u_{j} \}_{ j \in \N } $ {\it converges $ c $-strictly} in $ U( \Omega ) $ to $ u $,
symbolically $ u_{j} \gre{c} u $, if 
\begin{itemize}
\item
$ u_{j} \gre{|.|} u $ in $ U( \Omega ) $,
\item
$ \int_{ \Omega } c( E_{\dev} u_{j} ) \to \int_{ \Omega } c( E_{\dev} u ) $,
\item
$ \int_{ \Omega } c( E u_{j} ) \to \int_{ \Omega } c( E u ) $.
\end{itemize}
\end{enumerate}
\end{definition}
\item
First let us mention that a bounded sequence from $ U( \Omega ) $
contains a weakly convergent subsequence. 
This follows immediately from the corresponding results in $ BD( \Omega ) $ and $ L^{2}( \Omega ) $.
\item
For functions in $ U( \Omega ) $ outside $ LU( \Omega ) $, 
an approximation by smooth functions is not possible in the norm topology.
However, we may at least get an approximation in the $c$-strict topology.
We give this result in the form of \cite[Theorem~14.1.4]{AttouchButtazzoMichaille}:
\begin{theo}
\label{theo:c-strict-density}
Let $ \Omega \subset \R^{n} $ be a bounded Lipschitz domain 
and $ c : \M{n}{n}_{\sm} \to \R $ a non-negative convex function such that
\begin{itemize}
\item
there exist $ \alpha, \beta > 0 $ such that
for all $ X \in \M{n}{n}_{\sm} $
it holds 
\[ \alpha( |X| - 1 ) \le c(X) \le \beta( |X| + 1 ), \]
\item
the domain of its conjugate $ c^{*} $ is closed.
\end{itemize}
Then for every $ u \in U( \Omega ) $ 
there exists $ \{ u_{j} \}_{j \in \N} \subset C^{\i}( \Omega ; \R^{n} ) \cap LU( \Omega ) $ such that
\[ \gamma( u_{j} ) = \gamma( u ) 
\quad \mbox{and} \quad
u_{j} \gre{c} u \mbox{ in $ U( \Omega )$}. \]
\end{theo}
\begin{remark}
\label{remark:oglatoklepaj}
In Section~\ref{section:Hencky-limsup} we will use this approximation with the convex function $ c = \langle \cdot \rangle $
defined as
\[ \langle X \rangle := \sqrt{ 1 + |X|^{2} }. \]
It obviously fulfils the growth conditions and $ \langle \cdot \rangle^{*} $ has closed domain since
\[ \langle Y \rangle^{*} = \left\{ 
\begin{array}{cl} 
- \sqrt{ 1 - |Y|^{2} }, & |Y| \le 1, \\
\i, & |Y| > 1. 
\end{array}
\right. \]
\end{remark}
\item
In $ L^{p} $-spaces for $ p > 1 $, there exists a form of the Helmholtz decomposition 
(e.g., \cite[Section 2.3]{Kristensen94} or \cite[Example~3.14]{Grubb}).
We show that a similar result holds also in the space $U$, which appears to have been unnoticed so far. 
\item
We will need the following standard existence and regularity result for Poisson's equation.
\begin{theo}
\label{theo:Poisson}
Let $ \Omega \subset \R^{n} $ be a cube or a bounded open set with $ C^{1,1} $-boundary. 
For any $ f \in L^{2}( \Omega ) $ the Dirichlet problem for Poisson's equation 
\[ \triangle \phi = f, 
\quad \phi \in W^{1,2}_{0}( \Omega ), \]
has a unique weak solution $ \phi \in W^{2,2}( \Omega ) $. 
Moreover, there exists $ C > 0 $ such that
\[ \| \phi \|_{ W^{2,2} } \le C \| f \|_{ L^{2} }. \]
\end{theo} 
\item
For the proof we refer to \cite[Theorem~9.15]{GilbargTrudinger} for $ C^{1,1} $-domains 
and to \cite[Section 9.1]{WuYin} for cubes.
\begin{prop}
\label{prop:Hodge-U}
Let $ \Omega \subset \R^{n} $ be a bounded open set with $ C^{1,1} $-boundary or a cube. 
Then for every $ u \in U( \Omega ) $ there exist unique
\[ v \in U( \Omega ) \mbox{ with } \DIV v = 0 
\quad \mbox{and} \quad
\phi \in W^{1,2}_{0}( \Omega ) \cap W^{2,2}( \Omega ) \]
such that $ u = v + \nabla \phi $.
Therefore, we have a decomposition
\[ U( \Omega ) = ( \ker \DIV ) \oplus ( \im \nabla ) \]
into two closed subspaces where here 
\[ \nabla: W^{1,2}_{0}( \Omega ) \cap W^{2,2}( \Omega ) \to W^{1,2}( \Omega ; \R^{n} ). \]
\end{prop} 
\begin{proof}
Let us define a map
\[ P: U( \Omega ) \to U( \Omega ),
\quad
P(u) := \nabla \phi \]
with $ \phi \in W^{2,2}( \Omega ) $ being the unique weak solution of 
\[ \triangle \phi = \DIV u, 
\quad \phi \in W^{1,2}_{0}( \Omega ). \]
$P$ is linear and idempotent. According to Theorem~\ref{theo:Poisson}, also
\[ \| \nabla \phi \|_{U}
= \| \nabla \phi \|_{ L^{1} }
+ \| \nabla \nabla \phi \|_{ M }
+ \| \triangle \phi \|_{ L^{2} }
\le C_{1} \| \phi \|_{ W^{2,2} } 
\le C_{2} \| \DIV u \|_{ L^{2} } 
\le C_{2} \| u \|_{ U }. \]
Therefore, $P$ is a projection. By the definition, $ \im P \subset \nabla \big( W^{1,2}_{0}( \Omega ) \cap W^{2,2}( \Omega ) \big) $.
Since for every $ \phi \in W^{1,2}_{0}( \Omega ) \cap W^{2,2}( \Omega ) $ it holds $ P( \nabla \phi ) = \nabla \phi $, actually
\[ \im P = \nabla \big( W^{1,2}_{0}( \Omega ) \cap W^{2,2}( \Omega ) \big). \]
Moreover, from
\[ \DIV u = 0 \iff \nabla \phi = 0 \iff \phi = 0, \]
it follows that $ \ker P = \ker \DIV $. 
\end{proof}
\item 
According to Lemma~\ref{lemma:equi-int}, for bounded sequences in $ W^{1,p} $, $ p>1 $, 
there exists a modified sequence with $p$-equiintegrable gradients.
For a function from $ U( \Omega ) $, only a part of the symmetrized gradient has a higher integrability.
Still, we may get a similar result where we achieve 2-equiintegrability in that part.
\begin{lemma}
\label{lemma:equi-int-U}
Let $ \Omega \subset \R^{n} $ be an open bounded set with $ C^{1,1} $-boundary 
and let $ \{ u_{j} \}_{ j \in \N } $ be a bounded sequence in $ U( \Omega ) $. 
There exist a subsequence $ \{ u_{ j_{k} } \}_{ k \in \N } $ 
and a sequence $ \{ \tilde{u}_{k} \}_{ k \in \N } \subset U( \Omega ) $ such that 
\begin{itemize}
\item $ \{ ( \DIV \tilde{u}_{k} )^{2} \}_{ k \in \N } $ is equiintegrable,
\item $ \{ u_{j_{k}} - \tilde{u}_{k} \}_{ k \in \N} \subset W^{1,2}( \Omega ; \R^{n} ) $ and therefore $ E^{s} u_{j_{k}} = E^{s} \tilde{u}_{k} $,
\item $ \lim_{k \to \i} \big| \{ \nabla ( \tilde{u}_{k} - u_{ j_{k} } ) \ne 0 \} \cup \{ \tilde{u}_{k} \ne u_{ j_{k} } \} \big| = 0 $.
\end{itemize}
Moreover, if $ \{ u_{j} \}_{j \in \N} $ converges weakly, strictly or $c$-strictly to $ u $ in $ U( \Omega ) $, 
then the $ \tilde{u}_{k} $ can be chosen in such a way that $ \gamma( \tilde{u}_{k} ) = \gamma(u) $ 
and $ \{ \tilde{u}_{k} \}_{k \in \N} $ converges to $ u $ in $ U( \Omega ) $ in the same manner. 
\end{lemma} 
\begin{proof}
Let us decompose 
\[ u_{j} = v_{j} + \nabla \phi_{j} \]
according to Proposition~\ref{prop:Hodge-U}. Since $ \triangle \phi_{j} = \DIV u_{j} $,
$ \{ \phi_{j} \}_{ j \in \N } $ is a bounded sequence in $ W^{2,2}( \Omega ) $ by Theorem~\ref{theo:Poisson}.
Denote $ w_{j} := \nabla \phi_{j} $.
A suitable subsequence of $ \{ w_{j} \}_{ j \in \N } $ converges weakly in $ W^{1,2}( \Omega ; \R^{n} ) $ to some $w$. 
According to the Lemma~\ref{lemma:equi-int}, there exist a further subsequence
$ \{ w_{ j_{k} } \}_{k \in \N} $ and a sequence $ \{ \tilde{w}_{k} \}_{k \in \N} \subset w + W_{0}^{1,2}( \Omega ; \R^{n} ) $
such that 
\begin{itemize}
\item $ \tilde{w}_{k} \weakly w $ in $ W^{1,2}( \Omega ; \R^{n} ) $, 
\item $ \{ | \nabla \tilde{w}_{k} |^{2} \}_{ k \in \N } $ is equiintegrable,
\item $ \lim_{ k \to \i } | \{ w_{j_{k}} \ne \tilde{w}_{k} \} \cup \{ \nabla w_{j_{k}} \ne \nabla \tilde{w}_{k} \} | = 0 $.  
\end{itemize}
Define 
\[ \tilde{u}_{k} := v_{ j_{k} } + \tilde{w}_{k}. \]
It is a bounded sequence in $ U( \Omega ) $ that has the desired properties.
\item
As for supplement: If $ u_{j} \weakly u $, then also the related projections weakly converge,
and we continue as above. For the ($c$-)strict convergence we assess, employing Lipschitz continuity of $c$,
\begin{eqnarray*}
\left| \int_{ \Omega } c( E u_{j_{k}} ) - \int_{ \Omega } c( E \tilde{u}_{k} ) \right|
&  =  & \left| \int_{ \Omega } \big( c( \E u_{j_{k}}(x) ) - c( \E \tilde{u}_{k}(x) ) \big) \x \right| \\
& \le & \int_{ \Omega } L \big| \E u_{j_{k}}(x) - \E \tilde{u}_{k}(x) \big| \x \\
&  =  & \int_{ \{ \E w_{j_{k}} \ne \E \tilde{w}_{k} \} } L \big| \E w_{j_{k}}(x) - \E \tilde{w}_{k}(x) \big| \x.
\end{eqnarray*}
If $ k \to \i $, the last expression converges to 0 since $ \{ \E w_{j_{k}} - \E \tilde{w}_{k} \}_{ k \in \N } $ is bounded in $ L^{2} $
and is thus equiintegrable. Similarly for $ E_{\dev} $. 
\end{proof}
\begin{remark}
Young measures offer a possibility to describe weak convergence more precisely.
Namely, a highly oscillatory sequence may converge to a constant, 
which clearly does not contain any information about the members.
For integral functionals that we explore, gradient (or better symmetrized-gradient) Young measures should be considered.
While sufficient for $ p > 1 $, 
the theory must be generalized for $ p = 1 $ since beside oscillations also concentrations must be incorporated.
The concept dates back to the article \cite{DiPernaMajda} from DiPerna and Majda where they cope with a specific problem.
The frame-work for the general theory was set in \cite{AlibertBouchitte}. 
For the full-gradient case, the corresponding generalized gradient Young measures generated by sequences in $ BV $
were identified in \cite{KristensenRindler-Young}. 
Recently, also the symmetrized-gradient case was resolved, see \cite{DePhilippisRindler:17}.
The right analogue concept for the Hencky plasticity would contain measures generated by sequences in $U$.
By the decomposition lemma~\ref{lemma:equi-int-U}, the generating sequence may be taken to have 2-equiintegrable divergences. 
\end{remark}
\end{trivlist}
\section{Homogenization, vanishing hardening and commutability}
\label{section:Hencky-commutability}
\begin{trivlist}
\item
We now turn to the proofs of our main results. Having extended the notion of homo\-genization to functions with Hencky plasticity growth in Section~\ref{subsection:Hencky-setting}, we will focus on the homo\-genization of the energy functional for the zero-hardening case, which is the core of our analysis. In Section~\ref{section:Hencky-limsup} we construct a recovery sequence, drawing on known results for densities with linear growth, cf.\ \cite{KristensenRindler-Relax} and Theorem~\ref{theo:BD-Alberti}. The sequence will converge in the $ \langle \cdot \rangle $-strict topology so that we will be allowed to apply the Reshetnyak continuity theorem. We will have to handle the recession function with special attention. For that purpose we will derive a rather technical assessment on the Lipschitz constant.
\item
Regarding the $ \liminf $-inequality, 
the approach from \cite{AttouchButtazzoMichaille} with the slicing method of De Giorgi can be adapted for regular points, where, however, we will have to take care of the divergence with the help of Theorem~\ref{theo:Galdi}. Therefore, we will have to move the analysis form $ L^{1} $ to $ L^{q} $, $q > 1$. By Corollary~\ref{cor:BD-approx-diff-L1*} it is still possible to employ approximate differentiability. 
To control also the singular points, 
we will make use of our asymptotic convexity condition, which will enable the use of the results from \cite{DemengelQi}.
Let us mention that our analysis does not need extra regularity properties such as the domain of the convex conjugate of the homogenized density being closed, which is tacitly assumed in \cite{DemengelQi}.   
\item
Lastly, we will re-examine a special case: the relaxation problem. 
In \cite{KirchheimKristensen} new results regarding automatic convexity of 1-homogeneous functions with enough convex directions were shown.
Instead of the asymptotic convexity, we will impose a reasonable growth condition and show an analogous relaxation result.
\end{trivlist}
\subsection{Setting and homogenized density}
\label{subsection:Hencky-setting}
\begin{trivlist}
\item
First, let us concisely repeat the setting and explain the known facts. 
Throughout we will consider a Carath\'{e}odory function $ f : \R^{n} \times \M{n}{n}_{\sm} \to \R $ 
that is 
\begin{itemize}
\item 
$ \Y^{n} $-periodic in the first variable, 
\item 
has growth properties typical for the densities in the Hencky plasticity, 
i.e., there exist $ \alpha , \beta > 0 $ such that for all $ x \in \Omega $ and $ X \in \M{n}{n}_{\sm} $
\[ \alpha( | X_{\dev} | + ( \tr X )^{2} ) \le f( x , X ) \le \beta( | X_{\dev} | + ( \tr X )^{2} + 1 ). \]
\end{itemize}
For any $ \delta \ge 0 $ we define
\[ f^{(\delta)} : \R^{n} \times \M{n}{n}_{\sm} \to \R,
\quad
f^{(\delta)}( x , X ) := f( x , X ) + \delta | X_{\dev} |^{2}. \]
We will investigate the integral functionals $ \F_{\eps} $ and $ \F_{\eps}^{( \delta )} $, $ \delta \ge 0 $, on $ L^{1}( \Omega ; \R^{n} ) $
defined as
\[ \F_{\eps}(u) := \left\{ \begin{array}{ll}
\int_{ \Omega } f( \tfrac{x}{ \eps } , \E u(x) ) \x, & u \in LU( \Omega ; \R^{n} ), \\
\i, & {\rm else,}
\end{array} \right. \]
and
\[ \F^{(\delta)}_{\eps}(u) := \left\{ \begin{array}{ll}
\int_{ \Omega } f^{(\delta)}( \tfrac{x}{ \eps } , \E u(x) ) \x, & u \in W^{1,2}( \Omega ; \R^{n} ), \\
\i, & {\rm else}.
\end{array} \right. \]
\item
Obviously, for any $ u \in L^{1}( \Omega ; \R^{n} ) $ we have
\[ \F^{(\delta)}_{\eps}(u) \searrow \F^{(0)}_{\eps}(u)
\quad \mbox{as} \quad
\delta \to 0. \]
Then also
\[ \Gamma( L^{1} ) \mbox{-} \lim_{ \delta \to 0 } \F^{ (\delta) }_{ \eps }
= \lsc \F^{ (0) }_{ \eps }
= \lsc \F_{ \eps }. \]
The last equality follows from the fact that $ W^{1,2}( \Omega ; \R^{n} ) $, even $ C^{\i}( \overline{ \Omega } ; \R^{n} ) $,
is dense in $ ( LU( \Omega ; \R^{n} ) , \|.\|_{U} ) $, 
and can be proved by the same strategy as the one at the beginning of Subsection~\ref{subsection:U}.
\item
For $ \delta > 0 $ we may apply Proposition~\ref{prop:representation-formula} and Corollary~\ref{cor:hom-Garding}.
Hence, for every $ u \in L^{2}( \Omega ; \R^{n} ) $
\[ \Gamma( L^{2} ) \mbox{-} \lim_{ \eps \to 0 } \F^{ (\delta) }_{ \eps }(u)
= \F^{ (\delta) }_{ \hom }(u) \]
where 
\[ \F^{(\delta)}_{\hom}(u) := \left\{ \begin{array}{ll}
\int_{ \Omega } f^{(\delta)}_{ \hom }( \E u(x) ) \x, & u \in W^{1,2}( \Omega ; \R^{n} ), \\
\i, & {\rm else},
\end{array} \right. \]
has density
\[ f^{(\delta)}_{\hom}(X) 
= \inf_{ k \in \N } \inf_{ \varphi \in W^{1,2}_{0}( k \Y^{n} ; \R^{n} ) } \frac{1}{ k^{n} } \int_{ k \Y^{n} } f^{(\delta)}( x , X + \E \varphi(x) ) \x. \]
From the lower bound and Korn's inequality,
it follows that this family $ \Gamma $-converges even with respect to the $ L^{1} $-norm on the whole $ L^{1}( \Omega ; \R^{n} ) $. 
Schematically, so far we have
\begin{equation}
\label{eq:Hencky-Gamma-begin}
\begindc{\commdiag}[500]
\obj(0,2)[aa]{$ \F^{ (\delta) }_{\eps} $}
\obj(2,2)[bb]{$ \F^{ (0) }_{\eps} $}
\obj(0,0)[cc]{$ \F^{ (\delta) }_{\hom} $}
\mor{aa}{bb}{ptw. falling}
\mor{aa}{cc}{$ \Gamma $}
\enddc
\qquad \mbox{and} \qquad
\begindc{\commdiag}[500]
\obj(0,2)[aa]{$ \F^{ (\delta) }_{\eps} $}
\obj(2,2)[bb]{$ \lsc \F^{ (0) }_{\eps} = \lsc \F_{\eps}$}
\obj(0,0)[cc]{$ \F^{ (\delta) }_{\hom} $}
\mor{aa}{bb}{$ \Gamma $}
\mor{aa}{cc}{$ \Gamma $}
\enddc
\end{equation}
\item
We expect the density of the homogenized functional to have an analogous form also for $f$ 
(of course, with an additional singular term).
Therefore, we define
\[ f_{\hom}(X) 
:= \inf_{k \in \N} \inf_{ \varphi \in C_{c}^{\i}( k \Y^{n} ; \R^{n} ) } \frac{1}{ k^{n} } \int_{ k \Y^{n} } f( x , X + \E \varphi(x) ) \x. \]
By the reasoning in Subsection~\ref{subsection:U},
\[ f_{\hom}(X) 
= \inf_{k \in \N} \inf_{ \varphi \in LU_{0}( k \Y^{n} ) } \frac{1}{ k^{n} } \int_{ k \Y^{n} } f( x , X + \E \varphi(x) ) \x. \]
Let us explore the properties of the new function.
Since by \cite[Theorem 9.8]{Dacorogna}
\[ \inf_{ \varphi \in C_{c}^{\i}( k \Y^{n} ; \R^{n} ) } \int_{ k \Y^{n} } f( x , X + \E \varphi(x) ) \x
= \inf_{ \varphi \in C_{c}^{\i}( k \Y^{n} ; \R^{n} ) } \int_{ k \Y^{n} } f^{\qcls}( x , X + \E \varphi(x) ) \x, \]
it follows immediately $ ( f^{\qcls} )_{\hom} = f_{\hom} $.
This function is, as in the case with standard growth, (symmetric-)quasiconvex which can be argued with the regularization by hardening. 
Namely, $ f^{(\delta)} $ are symmetric-quasiconvex for $ \delta > 0 $,
and $ f^{(\delta)}_{\hom} $ is descending as $ \delta \searrow 0 $.
Moreover, 
\begin{eqnarray*} 
\inf_{ \delta > 0 } f^{(\delta)}_{\hom}(X)
& = & \inf_{ \delta > 0 } \inf_{k \in \N} \inf_{ \varphi \in C_{c}^{\i}( k \Y^{n} ; \R^{n} ) } \frac{1}{ k^{n} } \int_{ k \Y^{n} } f^{(\delta)}( x , X + \E \varphi(x) ) \x \\
& = & \inf_{k \in \N} \inf_{ \varphi \in C_{c}^{\i}( k \Y^{n} ; \R^{n} ) } \inf_{ \delta > 0 } \frac{1}{ k^{n} } \int_{ k \Y^{n} } f^{(\delta)}( x , X + \E \varphi(x) ) \x \\
& = & \inf_{k \in \N} \inf_{ \varphi \in C_{c}^{\i}( k \Y^{n} ; \R^{n} ) } \frac{1}{ k^{n} } \int_{ k \Y^{n} } f( x , X + \E \varphi(x) ) \x \\
& = & f_{\hom}( X ),
\end{eqnarray*}
where we employed the dominated convergence theorem (for arbitrary sequence $ \delta_{j} \searrow 0 $). 
Hence, again applying the same theorem, the function $ f_{\hom} $ is symmetric-quasiconvex.
\item
The typical ergodic formula will follow from the theory on subadditive processes from \cite{LichtMichaille}. 
Let $ {\cal B}_{b}( \R^{n} ) $ denote the set of all bounded Borel subsets of $ \R^{n} $.
For $ X \in \M{n}{n}_{\sm} $ and $ A \in {\cal B}_{b}( \R^{n} ) $, let us define
\[ m_{X}( A ) := 
\inf_{ \varphi \in C_{c}^{\i}( \Int A ; \R^{n} ) } \int_{ \Int A } f( x , X + \E \varphi(x) ) \x. \]
Then the mapping
\[ m_{X} : {\cal B}_{b}( \R^{n} ) \to \R \]
fulfils the assumptions of \cite[Theorem 2.1]{LichtMichaille}. 
Thus, for every open bounded convex set $ A $ and $ \eps_{k} \searrow 0 $ 
\begin{eqnarray*}
\lim_{ k \to \i } \frac{ m_{X}( \eps_{k}^{-1} A ) }{ | \eps_{k}^{-1} A | }
& = & \lim_{ k \to \i } \inf \left\{ 
\frac{1}{ | \eps_{k}^{-1} A | } \int_{ \eps_{k}^{-1} A } f( x , X + \E \varphi(x) ) \x:
\varphi \in C_{c}^{\i}( \eps_{k}^{-1} A ; \R^{n} ) \right\} \\
& = & f_{\hom}(X). 
\end{eqnarray*}
\item
Let us gather the results (and additional properties).
\begin{prop}
\label{prop:hom-density-U}
Let $ f : \R^{n} \times \M{n}{n}_{\sm} \to \R $ be a Carath\'{e}odory function that is $ \Y^{n} $-periodic in the first variable and
fulfils for some $ 0 < \alpha < \beta $
\[ \alpha( | X_{\dev} | + ( \tr X )^{2} ) \le f(x,X) \le \beta ( | X_{\dev} | + ( \tr X )^{2} + 1 ) \]
for a.e.~$ x \in \R^{n} $ and all $ X \in \M{n}{n}_{\sm} $. Then the homogenized function
\[ f_{\hom}(X) := \inf_{ k \in \N } \inf_{ \varphi \in C_{c}^{\i}( k \Y^{n} ; \R^{n} ) } 
\frac{1}{ k^{n} } \int_{ k \Y^{n} } f( x , X + \E \varphi(x) ) \x \]
is well-defined, satisfies the same growth condition, and for every open bounded convex set $ A $ and $ \eps_{k} \searrow 0 $ it holds
\[ f_{\hom}(X) 
= \lim_{ k \to \i } \inf \left\{ 
\frac{1}{ | \eps_{k}^{-1} A | } \int_{ \eps_{k}^{-1} A } f( x , X + \E \varphi(x) ) \x:
\varphi \in C_{c}^{\i}( \eps_{k}^{-1} A ; \R^{n} ) \right\}. \]
Moreover, $ f_{\hom} $ is symmetric-quasiconvex and
\[ ( f^{\qcls} )_{\hom} = f_{\hom}. \]
\end{prop}
%
\item 
Let us apply this to our question. Define the functionals
\begin{equation}
\label{eq:def-G}
\G(u) := \left\{ \begin{array}{ll}
\int_{ \Omega } f_{ \hom }( \E u(x) ) \x, & u \in LU( \Omega ; \R^{n} ), \\
\i, & {\rm else,}
\end{array} \right.
\end{equation}
and $ \G^{(0)} $ by reducing the domain of $ \G $ to $ W^{1,2}( \Omega ; \R^{n} ) $.
By the discussion above, for any $ \delta_{j} \searrow 0 $
we have the pointwise convergence $ ( f^{( \delta_{j} )} )_{ \hom } \searrow f_{ \hom } $.
It clearly follows
\[ \F^{(\delta_{j})}_{\hom}(u) \searrow \G^{(0)}(u) \]
for every $ u \in L^{1}( \Omega ; \R^{n} ) $.
Having a non-increasing sequence, it holds furthermore
\begin{equation}
\label{eq:Hencky-Gamma-delta}
\Gamma( L^{1} ) \mbox{-} \lim_{ \delta \to 0 } \F^{ (\delta) }_{ \hom }
= \lsc \G^{ (0) }
= \lsc \G 
\end{equation}
where for the second equality we argue as in Subsection~\ref{subsection:U}.
From $ \F^{(\delta)}_{\eps} \ge \F^{(0)}_{\eps} $ for every $ \delta > 0 $, it follows
\[ \F^{(\delta)}_{\hom} \ge \Gamma( L^{1} ) \mbox{-} \limsup_{ \eps \to 0 } \F^{(0)}_{\eps}. \]
The right-hand side is, being $ \Gamma $-$ \limsup $, lower semicontinuous. Hence, sending $ \delta \to 0 $ yields
\[ \lsc \G^{(0)} \ge \Gamma( L^{1} ) \mbox{-} \limsup_{ \eps \to 0 } \F^{(0)}_{\eps}. \]
Thus
\begin{equation}
\label{eq:Hencky-limsup}
\lsc \G
= \lsc \G^{(0)} 
\ge \Gamma( L^{1} ) \mbox{-} \limsup_{ \eps \to 0 } \F^{(0)}_{\eps}
\ge \Gamma( L^{1} ) \mbox{-} \limsup_{ \eps \to 0 } \F_{\eps}.
\end{equation}
\end{trivlist}
\subsection{$ \langle \cdot \rangle $-strict continuity and recovery sequences}
\label{section:Hencky-limsup}
\begin{trivlist}
\item
Our main objective in this paragraph is to prove Proposition~\ref{prop:c-strict-continuity}. 
Before doing so, we show how this result yields the $ \limsup $-inequality in Theorem~\ref{theo:Hencky-epilog}. 
We start with a couple of remarks on $f_{\dev}$ and its recession function $( f_{\dev} )^{\i}$
(defined in Proposition~\ref{prop:c-strict-continuity} or, for an arbitrary function, below). 
\begin{remark}
\label{remark:c-strict-continuity}
\begin{enumerate}
\item[]
\item
Since the deviatoric symmetric rank-one matrices span the whole $ \M{n}{n}_{\dev} $, 
the function $ f_{\dev} $ is globally Lipschitz in the second variable, i.e., there exists a constant $C$, depending only on $n$ and $\beta$,
such that
\[ | f_{\dev}( x , X ) - f_{\dev}( x , Y ) | \le C | X - Y | \]
for all $ x \in \Omega $ and $ X,Y \in \M{n}{n}_{\dev} $. See Lemma~\ref{lemma:BKK}.
\item
Consequently, the recession function is simply
\[ ( f_{\dev} )^{\i}( x_{0} , P_{0} ) = \limsup_{ t \to \i } \frac{ f_{\dev}( x_{0} , t P_{0} ) }{t}. \] 
For every $ P_{0} \in \M{n}{n}_{\dev} $ symmetric rank-one, this $ \limsup $ is even a limit (or a supremum)
due to the convexity in the direction of $ P_{0} $
\begin{eqnarray*} 
\limsup_{ t \to \i } \frac{ f_{\dev}( x_{0} , t P_{0} ) }{t} 
& = & \lim_{ t \to \i } \frac{ f_{\dev}( x_{0} , t P_{0} ) - f_{\dev}( x_{0} , 0 ) }{t} \\
& = & \sup_{ t > 0 } \frac{ f_{\dev}( x_{0} , t P_{0} ) - f_{\dev}( x_{0} , 0 ) }{t}. 
\end{eqnarray*}
\item
We will apply Proposition~\ref{prop:c-strict-continuity} to an $x$-independent function. 
The continuity assumption on $ ( f_{\dev} )^{\i} $ is in that case trivially fulfilled.
\end{enumerate}
\end{remark}
\begin{proof}[Proof of Theorem~\ref{theo:Hencky-epilog}, part~1: recovery sequence.] 
Let $ \G $ be as in \eqref{eq:def-G}. 
$ \G|_{ LU( \Omega ) } $ has a symmetric-quasiconvex density $ f_{\hom} $.
Therefore, it is by Proposition~\ref{prop:c-strict-continuity} continuous in the $ \langle \cdot \rangle $-strict topology
with 
\[ \overline{\G}(u) 
:= \int_{ \Omega } f_{\hom}( \E u(x) ) \x + \int_{\Omega} \big( ( f_{\hom} )_{\dev} \big)^{\i} \big( \tfrac{ d E^{s} u }{ d | E^{s} u |}(x) \big) \ d | E^{s} u |(x) \]
being its continuous extension to $ U( \Omega ) $ in this topology.
By Theorem~\ref{theo:c-strict-density}, for every $ u \in U( \Omega ) $ 
there exists a sequence $ \{ u_{j} \}_{j \in \N} \subset C^{\i}( \Omega ; \R^{n} ) \cap LU( \Omega ) $
such that 
\[ u_{j} \gre{ \langle \cdot \rangle } u \mbox{ in } U( \Omega ) 
\quad \mbox{and consequently} \quad
\G( u_{j} ) \to \overline{\G}(u). \]
Hence,
\begin{equation}
\label{eq:Hencky-density}
( \lsc \G )|_{ U( \Omega ) } \le \overline{ \G }.
\end{equation}
By Remark~\ref{remark:c-strict-continuity}, it holds $ \big( ( f_{\hom} )_{\dev} \big)^{\i} = ( f_{\hom} )^{\#}|_{ \M{n}{n}_{\dev} } $.
Therefore, by \eqref{eq:Hencky-density}
\[ \lsc \G \le \F_{\hom}. \]
Employing \eqref{eq:Hencky-limsup}, we arrive at 
\[ \Gamma( L^{1} ) \mbox{-} \limsup_{ \eps \to 0 } \F_{\eps} \le \F_{\hom}. \qedhere \]
\end{proof}
\item
For our proof of Proposition~\ref{prop:c-strict-continuity} we will adapt the strategy in \cite[Section 3]{KristensenRindler-Relax} to our purposes, where in adition we have to carefully handle the quadratic growth in the trace direction. According to Theorem~\ref{theo:c-strict-density} and Remark~\ref{remark:oglatoklepaj}, 
we may approximate every function in $ U( \Omega ) $ with smooth functions in the $ \langle \cdot \rangle $-strict topology.
Therefore, the following form of the Reshetnyak continuity theorem (see \cite[Theorem 5]{KristensenRindler-Relax})
is applicable:
\begin{theo}[Reshetnyak continuity theorem]
\label{theo:Reshetnyak}
Let $ f \in { \mathbf E }( \Omega ; \R^{N} ) $ and
\[ \mu_{j} \gre{*} \mu 
\quad \mbox{ in } M( \Omega ; \R^{N} ) 
\quad \mbox{and} \quad
\langle \mu_{j} \rangle ( \Omega ) \to \langle \mu \rangle ( \Omega ). \]
Then
\begin{align*} 
& \lim_{ j \to \i } \left[ \int_{ \Omega } f \left( x , \frac{ d \mu_{j}^{a} }{ d \L^{n} }(x) \right) \x 
+ \int_{ \Omega } f^{\i} \left( x , \frac{ d \mu_{j}^{s} }{ d | \mu_{j}^{s} | }(x)\right) \ d | \mu_{j}^{s} |(x) \right] = \\
& = \int_{ \Omega } f \left( x , \frac{ d \mu^{a} }{ d \L^{n} }(x) \right) \x 
+ \int_{ \Omega } f^{\i} \left( x , \frac{ d \mu^{s} }{ d | \mu^{s} | }(x)\right) \ d | \mu^{s} |(x). 
\end{align*}
\end{theo}
\item 
Let us explain the denotations from the theorem. 
The recession function of some function $ f : \Omega \times \R^{N} \to \R $ is defined as
\[ f^{\i}( x_{0} , X_{0} ) := \limsup_{ X \to X_{0}, \ t \to \i } \frac{ f( x_{0} , t X ) }{t}. \]
\item
A continuous function $ f : \Omega \times \R^{N} \to \R $ 
belongs to $ { \mathbf E }( \Omega ; \R^{N} ) $ if the function
\[ Tf : \Omega \times B_{1}(0) \to \R,
\quad 
( Tf )( x , \check{X} ) := ( 1 - | \check{X} | ) f( x , \tfrac{ \check{X} }{ 1 - | \check{X} | } ), \]
has a bounded continuous extension to $ \overline{ \Omega \times B_{1}(0) } $. 
For these functions the recession function is actually the limit
\[ f^{\i}( x_{0} , X_{0} ) = \lim_{ \scriptsize{ \begin{array}{c}  x \to x_{0}, \\ X \to X_{0}, \\ t \to \i \end{array} } } \frac{ f( x , t X ) }{t} \]
and agrees on $ \Omega \times \partial B_{1}(0) $ with the extension of $Tf$.
With functions from $ { \mathbf E }( \Omega ; \R^{N} ) $, we may approximate from below a large class of functions
as was shown in \cite[Lemma~2.3]{AlibertBouchitte}:
\begin{lemma}
\label{lemma:approximation-Alibert-Bouchitte}
Let $ f : \Omega \times \R^{N} \to \R $ be lower semicontinuous such that for some $ \alpha > 0 $ it holds
\[ f( x , X ) \ge - \alpha ( 1 + |X| ) \]
for all $ x \in \Omega $ and $ X \in \R^{n} $. 
There exists a non-decreasing sequence of functions $ \{ g_{k} \}_{ k \in \N } $ from $ {\mathbf E}( \Omega ; \R^{n} ) $ such that
\[ g_{k}( x , X ) \ge - \alpha ( 1 + |X| ),
\quad 
\sup_{k \in \N} g_{k} = f 
\quad \mbox{and} \quad
\sup_{k \in \N} g_{k}^{\i} = h_{f} \]
where
\[ h_{f}( x_{0} , X_{0} ) := \liminf \left\{ \frac{ f( x , t X ) }{t} : x \to x_{0}, \ X \to X_{0}, \ t \to \i \right\}. \]
\end{lemma}
\item
Since $ h_{f} $ will play a significant role in the approximation of the recession function, 
let us give a more detailed formula for our setting.
\item
Let $ f : \Omega \times \M{n}{n}_{\sm} \to \R $.
For $ P_{0} \in \M{n}{n}_{\dev} $ we may rewrite the definition in the following manner
\begin{eqnarray*} 
h_{f}( x_{0} , P_{0} ) 
& = & \liminf \left\{ \frac{ f ( x , t ( P + \tfrac{ \rho }{n} I ) ) }{t}  : 
	x \to x_{0}, \ P \to P_{0} \mbox{ in } \M{n}{n}_{\dev}, \ \rho \to 0, \ t \to \i \right\} \\
& = & \sup_{ k \in \N } \inf_{ ( x , P , \rho , t ) \in E_{ x_{0} , P_{0} , k} } \frac{ f ( x , t ( P + \tfrac{ \rho }{n} I ) ) }{t} 
\end{eqnarray*}
where
\[ E_{ x_{0} , P_{0} , k} := \left\{ ( x , X , \rho , t ) : 
| x - x_{0} | < \frac{1}{k} , \ P \in \M{n}{n}_{\dev}, | P - P_{0} | < \frac{1}{k} , | \rho | < \frac{1}{k} , t > k \right\}. \]
Notice that in $ - h_{-f} $ $ \liminf $ is replaced by $ \limsup $, i.e.,
\begin{eqnarray*} 
-h_{-f}( x_{0} , P_{0} ) 
& = & \limsup \left\{ \frac{ f ( x , t ( P + \tfrac{ \rho }{n} I ) ) }{t}  : 
	x \to x_{0}, \ P \to P_{0} \mbox{ in } \M{n}{n}_{\dev}, \ \rho \to 0, \ t \to \i \right\} \\
& = & \inf_{ k \in \N } \sup_{ ( x , P , \rho , t ) \in E_{ x_{0} , P_{0} , k} } \frac{ f ( x , t ( P + \tfrac{ \rho }{n} I ) ) }{t} 
\end{eqnarray*}
We wish to apply the Reshetnyak continuity theorem and Lemma~\ref{lemma:approximation-Alibert-Bouchitte}.
Therefore, we must carefully analyse the relationship between $ f^{\i} $ and $ ( f_{\dev} )^{\i} $,
and take into account the quadratic growth of $f$ in the trace direction.
For that reason, we prove a sort of Lipschitz continuity that enables us to compare values of finite and zero trace.   
\begin{lemma}
\label{lemma:Lipschitz-Hencky}
Let $ f : \M{n}{n}_{\sm} \to \R $ be symmetric-rank-one-convex and suppose 
\[ | f(X) | \le \beta( 1 + | X_{\dev} | + ( \tr X )^{2} ). \]
Then it fulfils the following local Lipschitz condition in the trace direction:
For any $ M, \varkappa \ge 1 $ and $ P \in \M{n}{n}_{ \dev } $,
it holds
\[ | f( P + \tfrac{ \rho }{n} I ) - f( P ) | \le 14 \beta n \sqrt{\varkappa} ( \sqrt{ | P | } + M ) |\rho| \]
for all $ \rho^{2} \le \varkappa ( | P | + M^{2} ) $.
\end{lemma}
\item 
The proof is based on Lemma~\ref{lemma:BKK}. 
It says that for a separately convex function $f$ on a ball $ B_{2r}( X ) $,
its Lipschitz constant on $ B_{r}( X ) $ does not exceed $ n \frac{ {\rm osc }( f , B_{2r}( X ) ) }{r} $.
\begin{proof}
Let $ P \in \M{n}{n}_{\dev} $ be arbitrary. Then
\begin{eqnarray*} 
{\rm osc} ( f , B( P , 2r ) )  
&  =  & \sup_{ X,Y \in B( P , 2r ) } | f(X) - f(Y) | \\
& \le & 2 \sup_{ X \in B( P , 2r ) } | f(X) | \\
& \le & 2 \beta( 1 + | P | + 2r + 4 r^{2} n ).
\end{eqnarray*}
Fix $r$ by $ r^{2} n = \varkappa( | P | + M^{2} ) $. Then 
\begin{eqnarray*} 
\tfrac{n}{r} {\rm osc} ( f , B( P , 2r ) )  
& \le & \tfrac{ 2 n \beta }{r} ( 1 + r^{2} n - M^{2} + 2r + 4 r^{2} n ) \\
& \le & 2 n \beta( 2 + 5 r n ) \\
& \le & 14 n^{3/2} \beta \sqrt{\varkappa} ( \sqrt{ | P | } + M ). 
\end{eqnarray*}
The claim follows as $ | \frac{ \rho }{n} I | = \frac{ \rho }{ \sqrt{n} } $. 
\end{proof}
\item
Let $f$ be as in Proposition~\ref{prop:c-strict-continuity}. 
For any $ M,K \in \N $ we set
\[ C_{M,K} := \{ X \in \M{n}{n}_{\sm} : | X_{\dev} | \ge K ( ( \tr X )^{2} - M^{2} ) \} \]
and define a lower and an upper bound for $f$
\[ \hat{f}_{M,K} \le f \le \check{f}_{M,K} \]
in the following way. 
Let us choose a continuous function $ \zeta_{M,K} : \M{n}{n}_{\sm} \to [0,1] $ such that 
\[ \zeta_{M,K}(X) = 1 \quad \mbox{for all $ X \in C_{M,K} $}
\quad \mbox{and} \quad
\zeta_{M,K}(X) = 0 \quad \mbox{for every $ X \not\in C_{M+1,K} $}. \]
Then define
\[ \hat{f}_{M,K}(x,X) := \zeta_{M,K}(X) f(x,X). \]
The function $ \hat{f}_{M,K} $ fulfils a linear growth condition since
\[ \hat{f}_{M,K}(x,X) \le 1_{ C_{M+1,K} }(X) \ f(x,X) \le \beta( 1 + 2 | X_{\dev} | + (M+1)^{2} ). \]
\item
We define the upper bound as
\[ \check{f}_{M,K}(x,X) := f(x,X) + \beta K^{2} \max \{ ( \tr X )^{2} - M^{2} - \tfrac{1}{K} | X_{\dev} | , 0 \}. \]
\begin{lemma}
\label{lemma:bound-h-M-K}
Suppose $f$ is as in Proposition~\ref{prop:c-strict-continuity}. For all $ x_{0} \in \Omega $, $ P_{0} \in \M{n}{n}_{\dev} $, $ M \ge 1 $
and $ K > 1 $ 
\[ - h_{ -\hat{f}_{M,K} }( x_{0} , P_{0} ) \le - h_{ - f_{\dev} }( x_{0} , P_{0} ) + \frac{ 14 \beta n }{ \sqrt{K} } | P_{0} | \]
and
\[ h_{ \check{f}_{M,K} }( x_{0} , P_{0} ) \ge h_{ f_{\dev} }( x_{0} , P_{0} ) - \frac{ 14 \beta n }{ \sqrt{K-1} } | P_{0} |. \]
\end{lemma}
\begin{proof}
Let us fix $ x_{0} \in \Omega $ and $ P_{0} \in \M{n}{n}_{\dev} $. We have to bound
\[ - h_{ - \hat{f}_{M,K} }( x_{0} , P_{0} ) 
= \inf_{ k \in \N } \sup_{ ( x , P , \rho , t ) \in E_{ x_{0} , P_{0} , k} } \frac{ \hat{f}_{M,K} ( x , t ( P + \tfrac{ \rho }{n} I ) ) }{t}. \] 
Take any $ ( x , P , \rho , t ) \in E_{ x_{0} , P_{0} , k } $.
\begin{itemize}
\item
For $ t ( P + \tfrac{\rho}{n} I ) \not\in C_{M+1,K} $ we have $ \hat{f}_{M,K}( x , t ( P + \tfrac{\rho}{n} I ) ) = 0 $.
\item
If $ t ( P + \tfrac{\rho}{n} I ) \in C_{M+1,K} $, then $ | t \rho | \le \sqrt{ \frac{1}{K} | t P | + (M+1)^{2} } $ and by Lemma~\ref{lemma:Lipschitz-Hencky}
\begin{eqnarray*} 
\hat{f}_{M,K}( x , t ( P + \tfrac{\rho}{n} I ) ) 
& \le & f( x , t ( P + \tfrac{\rho}{n} I ) ) \\
& \le & f( x , t P ) + 14 \beta n ( | t P |^{1/2} + M + 1 ) \sqrt{ \frac{1}{K} | t P | + (M+1)^{2} }.
\end{eqnarray*}
\end{itemize}
In both cases
\begin{eqnarray*} 
\frac{ \hat{f}_{M,K}( x , t ( P + \tfrac{\rho}{n} I ) ) }{t}
& \le & \frac{ f( x , t P ) }{t} + 14 \beta n \left( \frac{M+1}{ \sqrt{t} } + | P |^{1/2} \right) \sqrt{ \frac{1}{K} | P | + \frac{ (M+1)^{2} }{ \sqrt{t} } } \\
& \le & \frac{ f( x , t P ) }{t} + 14 \beta n \left( \frac{M+2}{ \sqrt{k} } + | P_{0} |^{1/2} \right) \sqrt{ \frac{1}{K} | P_{0} | + \frac{ (M+2)^{2} }{ \sqrt{k} } }.
\end{eqnarray*}
Hence
\[ - h_{ -\hat{f}_{M,K} }( x_{0} , P_{0} ) 
\le - h_{ -f_{\dev} }( x_{0} , P_{0} ) + \frac{ 14 \beta n }{ \sqrt{K} } | P_{0} |. \]
For
\[ h_{ \check{f}_{M,K} }( x_{0} , P_{0} ) 
= \sup_{ k \in \N } \inf_{ ( x , P , \rho , t ) \in E_{ x_{0} , P_{0} , k} } \frac{ \check{f}_{M,K} ( x , t ( P + \tfrac{ \rho }{n} I ) ) }{t}. \]
again choose arbitrary $ ( x , P , \rho , t ) \in E_{ x_{0} , P_{0} , k } $.
\begin{itemize}
\item
If $ t ( P + \tfrac{ \rho }{n} I ) ) \not\in C_{ M+1 , K-1 } $, it holds $ t^{2} \rho^{2} \ge \frac{1}{K-1} t|P| + ( M + 1 )^{2}  $ and therefore
\begin{eqnarray*} 
\check{f}_{M,K}( x , t ( P + \tfrac{ \rho }{n} I ) )
& \ge & \beta K^{2} ( t^{2} \rho^{2} - \tfrac{1}{K} t |P| - M^{2} ) \\
& \ge & \beta K^{2} ( \tfrac{1}{K-1} t |P| + ( M + 1 )^{2} - \tfrac{1}{K} t |P| - M^{2} ) \\
&  =  & \beta K^{2} ( \tfrac{1}{ K( K-1 ) } t |P| + 2 M + 1 ) \\
& \ge & \beta ( t |P| + 1 ) \\
& \ge & f( x , tP ).
\end{eqnarray*}
\item
If $ t ( P + \tfrac{ \rho }{n} I ) ) \in C_{ M+1 , K-1 } $, we assess as above
\begin{eqnarray*} 
\frac{ \check{f}_{M,K}( x , t ( P + \tfrac{ \rho }{n} I ) ) }{t}
& \ge & \frac{ f( x , t ( P + \tfrac{ \rho }{n} I ) ) }{t} \\
& \ge & \frac{ f( x , t P ) }{t} - 14 \beta n \left( \frac{M+1}{ \sqrt{t} } + | P |^{1/2} \right) \sqrt{ \frac{1}{K-1} | P | + \frac{ (M+1)^{2} }{ \sqrt{t} } } \\
& \ge & \frac{ f( x , t P ) }{t} - 14 \beta n \left( \frac{M+2}{ \sqrt{k} } + | P_{0} |^{1/2} \right) \sqrt{ \frac{1}{K-1} | P_{0} | + \frac{ (M+2)^{2} }{ \sqrt{k} } }.
\end{eqnarray*}
\end{itemize}
Therefore
\[ h_{ \check{f}_{M,K} }( x_{0} , P_{0} ) \ge h_{ f_{\dev} }( x_{0} , P_{0} ) - \frac{ 14 \beta n }{ \sqrt{K-1} } | P_{0} |. \qedhere \]
\end{proof} 
\item 
In the following three lemmas, the assumptions of Proposition~\ref{prop:c-strict-continuity} should hold.
\begin{lemma}
\label{lemma:c-strict-continuity-1}
The functional $ \F^{*} : U( \Omega ) \to \R $
\[ \F^{*}(u) := \int_{ \Omega } f( x , \E u(x) ) \x - \int_{ \Omega } h_{ -f_{\dev} }( x , \tfrac{ d E^{s} u }{ d | E^{s} u | }(x) ) \ d |E^{s} u|(x) \]
is upper semicontinuous with respect to the $ \langle \cdot \rangle $-strict topology.
\end{lemma}
\begin{proof}
Let us take any sequence $ u_{j} \gre{ \langle \cdot \rangle } u $ in $ U( \Omega ) $ 
and choose an arbitrary $ K \in \N $. 
Since the sequence $ \{ ( \DIV u_{j} )^{2} \}_{ j \in \N } $ is equiintegrable,
there exists $ M \in \N $ such that
\[ \int_{ \{ | \DIV u_{j} | > M \} } | \DIV u_{j}(x) |^{2} \x < \frac{1}{ K^{3} }. \]
For every $ j \in \N $ we split
\[ \int_{ \Omega } f( x , \E u_{j}(x) ) \x 
= \int_{ \{ \E u_{j} \in C_{M,K} \} } f( x , \E u_{j}(x) ) \x + \int_{ \{ \E u_{j} \not\in C_{M,K} \} } f( x , \E u_{j}(x) ) \x. \] 
For the first term obviously
\[ \int_{ \{ \E u_{j} \in C_{M,K} \} } f( x , \E u_{j}(x) ) \x
= \int_{ \{ \E u_{j} \in C_{M,K} \} } \hat{f}_{M,K}( x , \E u_{j}(x) ) \x
\le \int_{ \Omega } \hat{f}_{M,K}( x , \E u_{j}(x) ) \x. \]
By the definition $ \E u_{j}(x) \not\in C_{M,K} $ means $ | \E_{ \dev } u_{j}(x) | < K ( ( \DIV u_{j}(x) )^{2} - M^{2} ) $
and implies
\[ | \DIV u_{j}(x) | > M
\quad \mbox{and} \quad
f( x , \E u_{j}(x) ) 
\le \beta( 1 + 2 K ( \DIV u_{j}(x) )^{2} - K M^{2} ) )
\le 2 K \beta ( \DIV u_{j}(x) )^{2}. \]
Therefore, for every $ j \in \N $
\[ \int_{ \{ \E u_{j} \not\in C_{M,K} \} } f( x , \E u_{j}(x) ) \x
\le \int_{ \{ | \DIV u_{j} | > M \} } 2 K \beta ( \DIV u_{j}(x) )^{2} \x
\le \frac{ 2 \beta }{ K^{2} }. \]
The singular part $ E^{s} u $ is concentrated on $ \M{n}{n}_{\dev} $. 
Clearly $ - h_{ -f_{\dev} } \le - h_{ -\hat{f}_{M,K} } $ (on deviatoric matrices). Hence, for
\[ \hat{\F}_{M,K}(u) := 
\int_{ \Omega } \hat{f}_{M,K}( x , \E u(x) ) \x 
- \int_{ \Omega } h_{ - \hat{f}_{M,K} }( x , \tfrac{ d E^{s} u }{ d | E^{s} u | }(x) ) \ d |E^{s} u|(x) \]
we have $ \F^{*}( u_{j} ) \le \hat{\F}_{M,K}( u_{j} ) + \frac{ 2 \beta }{ K^{2} } $. 
Since $ \hat{f}_{M,K} $ grows linearly, 
we may approximate $ - \hat{f}_{M,K} $ from below according to Lemma~\ref{lemma:approximation-Alibert-Bouchitte} 
with a sequence $ \{ g_{k} \}_{ k \in \N } \subset {\mathbf E}( \Omega ; \M{n}{n}_{\sm} ) $.
Hence, for every $ k \in \N $
\begin{eqnarray*} \liminf_{ j \to \i } - \hat{\F}_{M,K}( u_{j} )
& \ge & \liminf_{ j \to \i } 
	\left( \int_{ \Omega } g_{k}( x , \E u_{j}(x) ) \x 
	+ \int_{ \Omega } g_{k}^{\i}( x , \tfrac{ d E^{s} u_{j} }{ d | E^{s} u_{j} | }(x) ) \ d | E^{s} u_{j} |(x) \right) \\
& \ge & \int_{ \Omega } g_{k}( x , \E u(x) ) \x 
	+ \int_{ \Omega } g_{k}^{\i}( x , \tfrac{ d E^{s} u }{ d | E^{s} u | }(x) ) \ d |E^{s} u|(x)
\end{eqnarray*}
since for $ g_{k} $ we may apply the Reshetnyak continuity theorem. Hence, by the monotone convergence theorem
\[ \liminf_{ j \to \i } - \hat{\F}_{M,K}( u_{j} ) \ge - \hat{\F}_{M,K}(u). \]
By gathering the results above we get
\[ \limsup_{ j \to \i } \F^{*}( u_{j} )
\le \limsup_{ j \to \i } \hat{\F}_{M,K}( u_{j} ) + \frac{ 2 \beta }{ K^{2} } 
\le \hat{\F}_{M,K}(u) + \frac{ 2 \beta }{ K^{2} }. \]
By Lemma~\ref{lemma:bound-h-M-K}
\begin{eqnarray*}
\hat{\F}_{M,K}(u) 
&  =  & \int_{ \Omega } \hat{f}_{M,K}( x , \E u(x) ) \x 
	- \int_{ \Omega } h_{ - \hat{f}_{M,K} }( x , \tfrac{ d E^{s} u }{ d | E^{s} u | }(x) ) \ d | E^{s} u |(x) \\
& \le & \int_{ \Omega } f( x , \E u(x) ) \x 
	- \int_{ \Omega } h_{ - f_{\dev} }( x , \tfrac{ d E^{s} u }{ d | E^{s} u | }(x) ) \ d | E^{s} u |(x) 
	+ \frac{ 14 \beta n }{ \sqrt{K} } | E^{s} u |( \Omega ).
\end{eqnarray*}
Altogether,
\[ \limsup_{ j \to \i } \F^{*}( u_{j} ) \le \F^{*}( u ) + \frac{ 2 \beta }{ K^{2} } + \frac{ 14 \beta n }{ \sqrt{K} } | E^{s} u |( \Omega ). \]
Since $ K $ was arbitrary, the upper semicontinuity follows. 
\end{proof}
\begin{lemma}
\label{lemma:c-strict-continuity-2}
The functional $ \F_{*} : U( \Omega ) \to \R $
\[ \F_{*}(u) := \int_{ \Omega } f( x , \E u(x) ) \x + \int_{ \Omega } h_{ f_{\dev} }( x , \tfrac{ d E^{s} u }{ d | E^{s} u | }(x) ) \ d |E^{s} u|(x) \]
is lower semicontinuous with respect to the $ \langle \cdot \rangle $-strict topology.
\end{lemma}
\begin{proof}
The proof is very similar to the previous one. Let us therefore just point out the differences.
Here we employ $ \check{f}_{M,K} $, which yields the lower bound
\begin{eqnarray*}
\int_{ \Omega } f( x , \E u_{j}(x) ) \x 
& \ge & \int_{ \Omega } \check{f}_{M,K}( x , \E u_{j}(x) ) \x 
	- \beta K^{2} \int_{ \{ \E u_{j} \not\in C_{M,K} \} } ( \DIV u_{j} )^{2} \x \\
& \ge & \int_{ \Omega } \check{f}_{M,K}( x , \E u_{j}(x) ) \x 
	- \frac{ \beta }{K}.
\end{eqnarray*}
Since on deviatoric matrices $ h_{ f_{\dev} } \ge h_{ \check{f}_{M,K} } $, we arrive at
$ \F_{*}( u_{j} ) \ge \check{\F}_{M,K}( u_{j} ) - \frac{ \beta }{K} $ where
\[ \check{\F}_{M,K}(u) := 
\int_{ \Omega } \check{f}_{M,K}( x , \E u(x) ) \x 
+ \int_{ \Omega } h_{ \check{f}_{M,K} }( x , \tfrac{ d E^{s} u }{ d | E^{s} u | }(x) ) \ d |E^{s} u|(x). \]
The function $ \check{f}_{M,K} $ meets the assumptions of Lemma~\ref{lemma:approximation-Alibert-Bouchitte}.
By the same argumentation and by Lemma~\ref{lemma:bound-h-M-K} we get
\[ \liminf_{ j \to \i } \F_{*}( u_{j} ) \ge \F_{*}( u ) - \frac{ \beta }{K} - \frac{ 14 \beta n }{ \sqrt{K-1} } | E^{s} u |( \Omega ). \qedhere \]
\end{proof}
\begin{lemma}
\label{lemma:c-strict-continuity-3}
For every $ x_{0} \in \Omega $ and symmetric rank-one matrix $ P_{0} \in \M{n}{n}_{\dev} $ 
\[ ( f_{\dev} )^{\i}( x_{0} , P_{0} ) = h_{ f_{\dev} }( x_{0} , P_{0} ) = - h_{ - f_{\dev} }( x_{0} , P_{0} ). \]
\end{lemma}
\item 
The proof is exactly the same as in \cite[Lemma~1]{KristensenRindler-Relax}.
\begin{proof}
Fix any $ x_{0} \in \Omega $ and any symmetric rank-one matrix $ P_{0} \in \M{n}{n}_{\dev} $.
Then
\[ \frac{ f_{\dev}( x , t P ) }{t}
= \frac{ f_{\dev}( x , t P ) - f_{\dev}( x , t P_{0} ) }{t}
+ \frac{ f_{\dev}( x , t P_{0} ) - f_{\dev}( x , 0 ) }{t}
+ \frac{ f_{\dev}( x , 0 ) }{t} \]
By Remark~\ref{remark:c-strict-continuity} we have
\[ \frac{ | f_{\dev}( x , t P ) - f_{\dev}( x , t P_{0} ) | }{t}
\le \frac{ C | t P - t P_{0} | }{t}
= C | P - P_{0} |, \]
and we know that the functions
\[ g_{t}(x) := \frac{ f_{\dev}( x , t P_{0} ) - f_{\dev}( x , 0 ) }{t} \]
make a monotonically increasing family with
\[ g_{t}(x) \nearrow ( f_{\dev} )^{\i}( x , P_{0} ) \]
for every $ x \in \Omega $. 
All $ g_{t} $ are continuous, and for $ ( f_{\dev} )^{\i}( \prostor , P_{0} ) $ continuity was an assumption (see Proposition~\ref{prop:c-strict-continuity}).
Therefore, we may apply Dini's Lemma. 
Hence,  
\[ g_{t} \nearrow ( f_{\dev} )^{\i}( \prostor , P_{0} ) \]
uniformly on $ \{ x \in \Omega : | x - x_{0} | \le \frac{1}{k} \} $ for every $ k \in \N $. 
Clearly we have also $ | \frac{ f_{\dev}( x , 0 ) }{t} | \le \frac{ \beta }{t} $.
Gathering all the estimates yields
\begin{eqnarray*}
h_{ f_{\dev} }( x_{0} , P_{0} )
&  =  & \liminf \left\{ \frac{ f_{\dev}( x , t P ) }{t} : x \to x_{0}, \ P \to P_{0}, \ t \to \i \right\} \\
&  =  & \liminf_{ x \to x_{0} } \ ( f_{\dev} )^{\i}( x , P_{0} ) \\
&  =  & ( f_{\dev} )^{\i}( x_{0} , P_{0} )
\end{eqnarray*}
as well as
\[ - h_{ -f_{\dev} }( x_{0} , P_{0} )
= \limsup_{ x \to x_{0} } \ ( f_{\dev} )^{\i}( x , P_{0} ) 
= ( f_{\dev} )^{\i}( x_{0} , P_{0} ). \qedhere \]
\end{proof}
\begin{proof}[Proof of Proposition~\ref{prop:c-strict-continuity}]
According to Theorem~\ref{theo:BD-Alberti},
for $ | E^{s} u | $-a.e.~$ x \in \Omega $, 
the Radon-Nikodym derivative $ \tfrac{ d E^{s} u }{ d | E^{s} u | }(x) $ is a symmetric rank-one matrix.
By Lemma~\ref{lemma:c-strict-continuity-3} 
the functionals from Lemmas~\ref{lemma:c-strict-continuity-1} and \ref{lemma:c-strict-continuity-2} coincide with $ \F $. 
\end{proof}
\end{trivlist}
\subsection{$\liminf$-inequality at zero hardening}
\label{section:Hencky-liminf}
\begin{trivlist}
\item
Now we turn our attention to the $ \liminf $-inequality. 
Let us take any $ u \in L^{1}( \Omega ; \R^{n} ) $,
and choose $ \eps_{j} \searrow 0 $ and $ u_{j} \to u $ in $ L^{1}( \Omega ; \R^{n} ) $.
Clearly, if 
$ \liminf_{ j \to \i } \F_{ \eps_{j} }( u_{j} ) = \i, $
there is nothing to be proved. If 
\[ \liminf_{ j \to \i } \F_{ \eps_{j} }( u_{j} ) < \i, \]
there exists a subsequence $ \{ j_{k} \}_{ k \in \N } $ such that
\[ \liminf_{ j \to \i } \F_{ \eps_{j} }( u_{j} ) = \lim_{ k \to \i } \F_{ \eps_{ j_{k} } }( u_{ j_{k} } ) \]
with all elements being finite. Hence, $ \{ u_{ j_{k} } \}_{ k \in \N } $ is bounded in $ LU( \Omega ) $,
and there exists a further (not relabelled) sequence that weakly converges in $ U( \Omega ) $ (see Subsection~\ref{subsection:U}).
Therefore, $ u \in U( \Omega ) $.
By Theorem~\ref{theo:BD-embeddings}, $ u_{j_{k}} \to u $ in $ L^{q}( \Omega ; \R^{n} ) $ for all $ 1 < q < \frac{n}{n-1} $.  
Moreover, we may also achieve that the measures 
\[ \mu_{k} := f( \tfrac{ \cdot }{ \eps_{j_{k}} } , \E u_{j_{k}}( \cdot ) ) \ \L^{n} \]
weakly-$*$ converge to some $ \mu $ in $ M( \Omega ; \R^{n} ). $ 
Let 
\[ \mu = g {\cal L}^{n} + \mu^{s} \]
be the decomposition according to the Radon-Nikodym theorem. 
Our aim will be to determine the derivative $g$ and the singular part $ \mu^{s} $,
as 
\begin{equation}
\label{eq:liminf}
\liminf_{ j \to \i } \F_{ \eps_{j} }( u_{j} )
= \lim_{ k \to \i } \mu_{k}( \Omega )
\ge \mu( \Omega )
= \int_{ \Omega } g(x) \x + \mu^{s}( \Omega ). 
\end{equation}
\item
The discussion above has shown that we may restrict ourselves to the following setting:
We consider arbitrary $ u \in U( \Omega ) $, $ \eps_{j} \searrow 0 $ 
and a bounded sequence $ \{ u_{j} \}_{ j \in \N } \subset LU( \Omega ) $ such that
\begin{itemize}
\item $ \lim_{ j \to \i } \F_{ \eps_{j} }( u_{j} ) $ exists,
\item $ u_{j} \weakly u $ in $ U( \Omega ) $ and $ u_{j} \to u $ in $ L^{q}( \Omega ; \R^{n} ) $ for a fixed $ 1 < q < \frac{n}{n-1} $, 
\item $ f( \tfrac{ \cdot }{ \eps_{j} } , \E u_{j}( \cdot ) ) \ \L^{n} =: \mu_{j} 
\weakstar \mu =: g {\cal L}^{n} + \mu^{s} $ in $ M( \Omega ; \R^{n} ). $
\end{itemize}
\item 
In the following two subsections we will bound $ g $ and $ \mu^{s} $ from below in regular and singular points, respectively. This leads to the conclusion of the proof of  Theorem~\ref{theo:Hencky-epilog}: 
\begin{proof}[Proof of Theorem~\ref{theo:Hencky-epilog}, part~2.]
The $\liminf$-inequality in Theorem~\ref{theo:Hencky-epilog} is a direct consequence of \eqref{eq:liminf}, Lemma~\ref{lemma:liminf-regular} and Lemma~\ref{lemma:liminf-singular}. 
\end{proof}
\item
Also the proof of Theorem~\ref{theo:Hencky-commutability} is now a direct consequence of our considerations. 
\begin{proof}[Proof of Theorem~\ref{theo:Hencky-commutability}.]
The upper and the left-hand arrows were derived in Section~\ref{subsection:Hencky-setting}
and depicted in \eqref{eq:Hencky-Gamma-begin}.
Both right-hand arrows follow from Theorem~\ref{theo:Hencky-epilog} 
and general properties of {$ \Gamma $-convergence}.
Finally, \eqref{eq:Hencky-Gamma-delta}, \eqref{eq:Hencky-limsup} and \eqref{eq:Hencky-density} imply
\[ \Gamma( L^{1} ) \mbox{-} \lim_{ \delta \to 0 } \F^{(\delta)}_{\hom} = \F_{\hom}. \qedhere \] 
\end{proof}
\end{trivlist}
\subsubsection{Regular points}
\begin{lemma}
\label{lemma:liminf-regular}
For a.e.~$ x_{0} \in \Omega $, it holds
\[ g( x_{0} ) \ge f_{ \hom }( \E u( x_{0} ) ). \]
\end{lemma}
\begin{trivlist}
\item 
We will follow the strategy of the proofs of \cite[Propositions~11.2.3~and~12.3.2]{AttouchButtazzoMichaille}.
\end{trivlist}
\begin{proof}
By the Besicovitch derivation theorem~\ref{theo:Besicovitch}, for a.e.~$ x_{0} \in \Omega $ we have
\begin{equation} 
\label{eq:Besicovitch}
g( x_{0} ) = \lim_{ \rho \to 0 } \frac{ \mu( B_{\rho}( x_{0} ) ) }{ | B_{\rho}( x_{0} ) | }. 
\end{equation}
For all but countable many $ \rho > 0 $ it holds 
\[ \mu( B_{\rho}( x_{0} ) ) = \lim_{ j \to \i } \mu_{j}( B_{\rho}( x_{0} ) ). \]
Therefore, it must be shown for a.e.~$ x_{0} \in \Omega $
\[ \lim_{ \rho \to 0 } \lim_{ j \to \i }
\frac{ \mu_{j}( B_{\rho}( x_{0} ) ) }{ | B_{\rho}( x_{0} ) | } \ge f_{\hom}( \E u( x_{0} ) ) \]
(whereby we exclude the exceptional sequence of $ \rho $'s). 
\item
Let us take and fix any $ x_{0} $ where the formula~\eqref{eq:Besicovitch} holds 
and where the function $u$ is $ L^{q} $-differentiable (see Corollary~\ref{cor:BD-approx-diff-L1*}), and define
\[ \tilde{u}(x) := u( x_{0} ) + \nabla u( x_{0} ) \ ( x - x_{0} ). \]
We may also suppose $ \E \tilde{u} = \E \tilde{u}( x_{0} ) = \sm \nabla u( x_{0} ) $.
Our strategy is to approximate $ u $ with $ \tilde{u} $ and to use the slicing method of De Giorgi. Therefore,
choose any $ \nu \in \N $ and $ 0 < \lambda < 1 $ and define
\[ \rho_{0} := \lambda \rho
\qquad \mbox{and} \qquad
B_{i} := B_{ \rho_{0} + \frac{i}{ \nu }( \rho - \rho_{0} ) }( x_{0} ),
\quad i = 0 , \ldots , \nu. \]
Furthermore, we take for every $ i = 1 , \ldots , \nu $ also cut-off functions $ \varphi_{i} \in C_{c}^{\i}( B_{i} ) $
such that
\[ 0 \le \varphi_{i} \le 1,
\quad
\varphi_{i} = 1 \mbox{ in } B_{i-1}
\quad \mbox{and} \quad
\| \nabla \varphi_{i} \|_{ L^{\i} } \le \frac{ 2 \nu }{ \rho - \rho_{0} }. \]
Let
\[ \tilde{u}_{j,i} := \tilde{u} + \varphi_{i}( u_{j} - \tilde{u} ) \in L^{1}( \Omega ; \R^{n} ) . \]
Because of
\[ \E \tilde{u}_{j,i} 
= ( 1 - \varphi_{i} ) \E \tilde{u} + \varphi_{i} \E u_{j} + \nabla \varphi_{i} \odot ( u_{j} - \tilde{u} ), \]
the functions $ \tilde{u}_{j,i} $ lie in $ LD( \Omega ) $, but perhaps not in $ LU( \Omega ) $ since 
\[ \DIV \tilde{u}_{j,i} 
= ( 1 - \varphi_{i} ) \DIV \tilde{u} + \varphi_{i} \DIV u_{j} + \nabla \varphi_{i} \cdot ( u_{j} - \tilde{u} ), \]
and the last term in general lies only in $ L^{ \frac{n}{n-1} } $. 
We will correct this with the results of Bogovskii from Theorem~\ref{theo:Galdi}.
For that reason define
\[ \zeta_{j,i} := 
\frac{1}{ | B_{i} \setminus B_{i-1} | } \int_{ B_{i} \setminus B_{i-1} } 
\nabla \varphi_{i}(x) \cdot ( u_{j}(x) - \tilde{u}(x) ) \x. \]
By Theorem~\ref{theo:Galdi} there exist $ z_{j,i} \in W^{1,q}_{0}( B_{i} \setminus \overline{ B_{i-1} } ) $ such that
\[ \DIV z_{j,i} = - \nabla \varphi_{i} \cdot ( u_{j} - \tilde{u} ) + \zeta_{j,i} \]
with
\[ \| z_{j,i} \|_{ W^{1,q}( B_{i} \setminus \overline{ B_{i-1} } ) } 
\le C \| \nabla \varphi_{i} \cdot ( u_{j} - \tilde{u} ) \|_{ L^{ q }( B_{i} \setminus \overline{ B_{i-1} } ) }
\le \frac{ 2 C \nu }{ \rho - \rho_{0} } \| u_{j} - \tilde{u} \|_{ L^{ q }( B_{i} \setminus \overline{ B_{i-1} } ) }. \]
We take such constant $C$ that the inequality holds for all $i$. 
It is scaling and translation invariant, so we may transfer the situation to $ B_{1}(0) $.
Therefore, $C$ does not depend on $ \rho $. 
Although it depends on $ \nu $ and $ \lambda $, this will not cause any troubles since we will first send $ j \to \i $. 
Now define $ u_{j,i} := \tilde{u}_{j,i} + z_{j,i} \in LU( \Omega ) $.
It is elementary to see that 
\[ u_{j,i} - \tilde{u} = \varphi_{i}( u_{j} - \tilde{u} ) + z_{j,i} \in LU_{0}( B_{ \rho }( x_{0} ) ). \]
(We refer to \cite{Jesenko:16} for a detailed argument.) For every $ i = 1 , \ldots , \nu $
\begin{eqnarray*} 
&     & \!\!\!\!\!\!\!\!\!\!
		f_{\hom}( \E u( x_{0} ) ) \\
&  =  & \lim_{ j \to \i } \inf \left\{ 
		\frac{1}{ | \frac{1}{ \eps_{j} } B_{ \rho }( x_{0} ) | } 
		\int_{ \frac{1}{ \eps_{j} } B_{ \rho }( x_{0} ) } f( x , \E u( x_{0} ) + \E \varphi(x) ) \x :
		\varphi \in LU_{0}( \tfrac{1}{ \eps_{j} } B_{ \rho }( x_{0} ) , \R^{n} )
		\right\} \\
&  =  & \lim_{ j \to \i } \inf \left\{ 
		\frac{1}{ | B_{ \rho }( x_{0} ) | } 
		\int_{ B_{ \rho }( x_{0} ) } f( \tfrac{x}{ \eps_{j} } , \E u( x_{0} ) + \E \varphi(x) ) \x :
		\varphi \in LU_{0}( B_{ \rho }( x_{0} ) , \R^{n} )
		\right\} \\		
& \le & \liminf_{ j \to \i } \frac{1}{ | B_{ \rho }( x_{0} ) | }
		\int_{ B_{ \rho }( x_{0} ) } f \big( \tfrac{x}{ \eps_{j} } , \E u_{j,i}(x) \big) \x,	
\end{eqnarray*}
and therefore also
\[ f_{\hom}( \E u( x_{0} ) ) 
\le \liminf_{ j \to \i } \frac{1}{ \nu } \sum_{ i=1 }^{ \nu } \frac{1}{ | B_{ \rho }( x_{0} ) | }
		\int_{ B_{ \rho }( x_{0} ) } f \big( \tfrac{x}{ \eps_{j} } , \E u_{j,i}(x) \big) \x. \]
For every $i$ we split
\begin{align*}
&    \int_{ B_{ \rho }( x_{0} ) } f \big( \tfrac{x}{ \eps_{j} } , \E u_{j,i}(x) \big) \x \\
& = \ \int_{ B_{i-1} } f \big( \tfrac{x}{ \eps_{j} } , \E u_{j}(x) \big) \x 
+ \int_{ B_{i} \setminus B_{i-1} } f \big( \tfrac{x}{ \eps_{j} } , \E u_{j,i}(x) \big) \x 
+ \int_{ B_{ \rho }( x_{0} ) \setminus B_{i} } f \big( \tfrac{x}{ \eps_{j} } , \E u( x_{0} ) \big) \x \\
& =: \  I_{j,i}^{(1)} + I_{j,i}^{(2)} + I_{j,i}^{(3)}. 
\end{align*}
The first term can be bounded simply by
\[ I_{j,i}^{(1)} \le \int_{ B_{ \rho }( x_{0} ) } f \big( \tfrac{x}{ \eps_{j} } , \E u_{j}(x) \big) \x \]
and the last one by
\[ I_{j,i}^{(3)}
\le \beta n ( 1 - \lambda ) | B_{ \rho }( x_{0} ) | \big( 1 + | \E_{ \dev } u( x_{0} ) | + | \DIV u( x_{0} ) |^{2} \big). \]
The second term we bound by the upper bound on $f$
\[ f \big( \tfrac{x}{ \eps_{j} } , \E u_{j,i}(x) \big) 
\le \beta( 1 + | \E_{ \dev } u_{j,i}(x) | + ( \DIV u_{j,i}(x) )^{2} ). \]
First,
\begin{eqnarray*}
| \E_{ \dev } u_{j,i} | 
&  =  & | ( 1 - \varphi_{i} ) \E_{ \dev }  \tilde{u} + \varphi_{i} \E_{ \dev }  u_{j} + \dev \sm( \nabla \varphi_{i} \otimes ( u_{j} - \tilde{u} ) ) + \E_{ \dev } z_{j,i} | \\
& \le & | \E_{ \dev }  \tilde{u} | + | \E_{ \dev }  u_{j} | + | \nabla \varphi_{i} | | u_{j} - \tilde{u} | + | \E_{ \dev } z_{j,i} |.
\end{eqnarray*}
First we bound the last two terms
\[ \int_{ B_{i} \setminus B_{i-1} } | \nabla \varphi_{i}(x) | | u_{j}(x) - \tilde{u}(x) | \x
\le \frac{ 2 \nu | B_{i} \setminus B_{i-1} |^{1/q'} }{ ( 1 - \lambda ) \rho } 
\left( \int_{ B_{i} \setminus B_{i-1} } | u_{j}(x) - \tilde{u}(x) |^{q} \x \right)^{1/q} \]
and 
\begin{eqnarray*}
\int_{ B_{i} \setminus B_{i-1} } | \E_{ \dev } z_{j,i}(x) | \x
& \le & | B_{i} \setminus B_{i-1} |^{1/q'} \left( \int_{ B_{i} \setminus B_{i-1} } | \E_{ \dev } z_{j,i}(x) |^{q} \x \right)^{1/q} \\
& \le & \frac{ 2 C \nu | B_{i} \setminus B_{i-1} |^{1/q'} }{ ( 1 - \lambda ) \rho } 
		\left( \int_{ B_{i} \setminus B_{i-1} } | u_{j}(x) - \tilde{u}(x) |^{q} \x \right)^{1/q}.
\end{eqnarray*}
($ q' $ stands for the H\"{o}lder conjugate of $q$.)
Thus, by applying $ | B_{i} \setminus B_{i-1} | \le \frac{ n ( 1 - \lambda ) }{ \nu } | B_{ \rho }( x_{0} ) | $, we arrive at
\begin{align*}
& \int_{ B_{i} \setminus B_{i-1} } | \E_{ \dev } u_{j,i}(x) | \x \\
& \le \ | B_{i} \setminus B_{i-1} | | \E_{ \dev }  u( x_{0} ) |
        + \int_{ B_{i} \setminus B_{i-1} } | \E_{ \dev } u_{j}(x) | \x + \\
&     \ + \frac{ 2 ( 1 + C ) \nu^{1/q} }{ ( 1 - \lambda )^{ 1/q }  } n^{1/q'} | B_{ \rho }( x_{0} ) |^{1/q'}
		\left( \int_{ B_{i} \setminus B_{i-1} }  \frac{ | u_{j}(x) - \tilde{u}(x) |^{q} }{ \rho^{q} } \x \right)^{1/q}. 
\end{align*}
Also
\[ ( \DIV u_{j,i} )^{2}
= \big( ( 1 - \varphi_{i} ) \DIV \tilde{u} + \varphi_{i} \DIV u_{j} + \zeta_{j,i} \big)^{2}
\le 3 ( \DIV \tilde{u} )^{2} + 3 ( \DIV u_{j} )^{2} + 3 \zeta_{j,i}^{2}, \]
and therefore,
\begin{eqnarray*}
\int_{ B_{i} \setminus B_{i-1} } ( \DIV u_{j,i}(x) )^{2} \x 
& \le & 3 | B_{i} \setminus B_{i-1} | ( \DIV u( x_{0} ) )^{2} + \\
&     & + 3 \int_{ B_{i} \setminus B_{i-1} }( \DIV u_{j}(x) )^{2} \x
		+ 3 \frac{ n ( 1 - \lambda ) }{ \nu } | B_{ \rho }( x_{0} ) | \zeta_{j,i}^{2}.
\end{eqnarray*}
For a bound for the last term, we proceed as above
\begin{eqnarray*}
| \zeta_{j,i} | 
& \le & \frac{1}{ | B_{i} \setminus B_{i-1} | } \int_{ B_{i} \setminus B_{i-1} } | \nabla \varphi_{i}(x) | | u_{j}(x) - \tilde{u}(x) | \x \\
& \le & \frac{ 2 \nu }{ ( 1 - \lambda ) | B_{i} \setminus B_{i-1} |^{1/q} } 
		\left( \int_{ B_{i} \setminus B_{i-1} }  \frac{ | u_{j}(x) - \tilde{u}(x) |^{q} }{ \rho^{q} } \x \right)^{1/q} \\
& \le & \frac{ 2 \nu^{ 1 + 1/q } }{ n^{1/q} \lambda^{n/q} ( 1 - \lambda )^{ 1 + 1/q } } 
		\left( \frac{1}{ | B_{ \rho }( x_{0} ) | } \int_{ B_{i} \setminus B_{i-1} }  \frac{ | u_{j}(x) - \tilde{u}(x) |^{q} }{ \rho^{q} } \x \right)^{1/q},
\end{eqnarray*}
where we used $ | B_{i} \setminus B_{i-1} | \ge \frac{ n \lambda^{n} ( 1 - \lambda ) }{ \nu } | B_{ \rho }( x_{0} ) | $. Then
\[ \sum_{i=1}^{\nu} | \zeta_{j,i} | 
\le \frac{ \nu^{2} }{ n^{1/q} \lambda^{n/q} ( 1 - \lambda )^{ 1 + 1/q } } 
		\left( \frac{1}{ | B_{ \rho }( x_{0} ) | } \int_{ B_{ \rho }( x_{0} ) }  \frac{ | u_{j}(x) - \tilde{u}(x) |^{q} }{ \rho^{q} } \x \right)^{1/q}. \]
Altogether for the second term
\begin{eqnarray*}
I_{j,i}^{(2)}
& \le & 3 \beta \Bigg[ | B_{i} \setminus B_{i-1} | \big( 1 + | \E_{ \dev }  u( x_{0} ) | + ( \DIV  u( x_{0} ) )^{2} \big) + \\
&     & + \int_{ B_{i} \setminus B_{i-1} } \big( | \E_{ \dev } u_{j}(x) | + ( \DIV u_{j}(x) )^{2} \big) \x  + \\
&     & + \frac{ 2 ( 1 + C ) \nu }{ ( 1 - \lambda )^{ 1/q }  } n^{1/q'} | B_{ \rho }( x_{0} ) |
		\left( \frac{1}{ | B_{ \rho }( x_{0} ) | } 
		\int_{ B_{i} \setminus B_{i-1} }  \frac{ | u_{j}(x) - \tilde{u}(x) |^{q} }{ \rho^{q} } \x \right)^{1/q} + \\
&     & + \frac{ n ( 1 - \lambda ) }{ \nu } | B_{ \rho }( x_{0} ) | \zeta_{j,i}^{2} \Bigg]
\end{eqnarray*}
and
\begin{align*}
&     \sum_{i=1}^{ \nu } I_{j,i}^{(2)} \\
& \le \ 3 \beta \Bigg[ n ( 1 - \lambda ) | B_{ \rho }( x_{0} ) | \big( 1 + | \E_{ \dev } u( x_{0} ) | + | \DIV u( x_{0} ) |^{2} \big) + \\
& \qquad + \int_{ B_{ \rho }( x_{0} ) } \frac{ f \big( \tfrac{x}{ \eps_{j} } , \E u_{j}(x) \big) }{ \alpha } \x + \\
& \qquad + \frac{ 2 ( 1 + C ) \nu }{ ( 1 - \lambda )^{ 1/q }  } ( n \nu )^{1/q'} | B_{ \rho }( x_{0} ) |
		\left( \frac{1}{ | B_{ \rho }( x_{0} ) | } 
		\int_{ B_{ \rho }( x_{0} ) }  \frac{ | u_{j}(x) - \tilde{u}(x) |^{q} }{ \rho^{q} } \x \right)^{1/q} + \\
& \qquad + \frac{ n ( 1 - \lambda ) }{ \nu } | B_{ \rho }( x_{0} ) | 
		\frac{ \nu^{4} }{ n^{2/q} \lambda^{2n/q} ( 1 - \lambda )^{ 2 + 2/q } } 
		\left( \frac{1}{ | B_{ \rho }( x_{0} ) | } \int_{ B_{ \rho }( x_{0} ) }  \frac{ | u_{j}(x) - \tilde{u}(x) |^{q} }{ \rho^{q} } \x \right)^{2/q} \Bigg].
\end{align*}
Now 
\begin{align*}
& \frac{1}{ \nu } \sum_{ i = 1 }^{ \nu } 
		\frac{ I_{j,i}^{(1)} + I_{j,i}^{(2)} + I_{j,i}^{(3)}  }{ | B_{ \rho }( x_{0} ) | } \\
& \le \ \frac{1}{ | B_{ \rho }( x_{0} ) | } \int_{ B_{ \rho }( x_{0} ) } 
		\left( 1 + \tfrac{ 3 \beta }{ \alpha \nu } \right) f \big( \tfrac{x}{ \eps_{j} } , \E u_{j}(x) \big) \x + \\
& \qquad + 4 \beta n ( 1 - \lambda )  \big( 1 + | \E_{ \dev } u( x_{0} ) | + | \DIV u( x_{0} ) |^{2} \big) + \\
& \qquad + \frac{ 2 ( 1 + C ) }{ ( 1 - \lambda )^{ 1/q }  } ( n \nu )^{1/q'} 
		\left( \frac{1}{ | B_{ \rho }( x_{0} ) | } 
		\int_{ B_{ \rho }( x_{0} ) }  \frac{ | u_{j}(x) - \tilde{u}(x) |^{q} }{ \rho^{q} } \x \right)^{1/q} + \\
& \qquad + \frac{ \nu^{2} n^{ 1 - 2/q } }{ \lambda^{2n/q} ( 1 - \lambda )^{ 1 + 2/q } } 
		\left( \frac{1}{ | B_{ \rho }( x_{0} ) | } \int_{ B_{ \rho }( x_{0} ) }  \frac{ | u_{j}(x) - \tilde{u}(x) |^{q} }{ \rho^{q} } \x \right)^{2/q}.
\end{align*}
We now send $ j \to \i $ and use in the last two terms that $ u_{j} \to u $ in $ L^{q}( \Omega ; \R^{n} ) $. Hence,
\begin{eqnarray*}
f_{\hom}( \mathfrak{E} u( x_{0} ) ) 
& \le & \liminf_{ j \to \i } \frac{1}{ \nu } \sum_{ i = 1 }^{ \nu } 
		\frac{ I_{j,i}^{(1)} + I_{j,i}^{(2)} + I_{j,i}^{(3)}  }{ | B_{ \rho }( x_{0} ) | }  \\
& \le & \left( 1 + \tfrac{ 3 \beta }{ \alpha \nu } \right)
		\lim_{ j \to \i } \frac{ \mu_{j}( B_{ \rho }( x_{0} ) ) }{ | B_{ \rho }( x_{0} ) | } + \\
&     & + 4 \beta n ( 1 - \lambda )  \big( 1 + | \E_{ \dev } u( x_{0} ) | + | \DIV u( x_{0} ) |^{2} \big) \\
&     & + \frac{ 2 ( 1 + C ) }{ ( 1 - \lambda )^{ 1/q }  } ( n \nu )^{1/q'} 
		\left( \frac{1}{ | B_{ \rho }( x_{0} ) | } 
		\int_{ B_{ \rho }( x_{0} ) }  \frac{ | u(x) - \tilde{u}(x) |^{q} }{ \rho^{q} } \x \right)^{1/q} + \\
&     & + \frac{ \nu^{2} n^{ 1 - 2/q } }{ \lambda^{2n/q} ( 1 - \lambda )^{ 1 + 2/q } } 
		\left( \frac{1}{ | B_{ \rho }( x_{0} ) | } \int_{ B_{ \rho }( x_{0} ) }  \frac{ | u(x) - \tilde{u}(x) |^{q} }{ \rho^{q} } \x \right)^{2/q}.
\end{eqnarray*} 
Sending also $ \rho \to 0 $ (excluding countably many) and applying $ L^{q} $-differentiability of $u$ yields
\begin{eqnarray*}
f_{\hom}( \mathfrak{E} u( x_{0} ) ) 
& \le & \left( 1 + \tfrac{ 3 \beta }{ \alpha \nu } \right)
		\lim_{ \rho \to 0 } \lim_{ j \to \i } \frac{ \mu_{j}( B_{ \rho }( x_{0} ) ) }{ | B_{ \rho }( x_{0} ) | } + \\
&     & + 4 \beta n ( 1 - \lambda )  \big( 1 + | \E_{ \dev } u( x_{0} ) | + | \DIV u( x_{0} ) |^{2} \big).
\end{eqnarray*} 
Since $ \lambda < 1 $ and $ \nu \in \N $ were arbitrary, we get
\[ f_{\hom}( \mathfrak{E} u( x_{0} ) ) 
\le \lim_{ \rho \to 0 } \lim_{ j \to \i } \frac{ \mu_{j}( B_{ \rho }( x_{0} ) ) }{ | B_{ \rho }( x_{0} ) | }. \qedhere \]
\end{proof}
\begin{remark}
\label{remark:Hencky-LU}
Until now we have not made use of the asymptotic convexity assumption (\ref{eq:asymptotic-convexity}). By \eqref{eq:liminf} and Lemma~\ref{lemma:liminf-regular},
we have for every $ u \in U( \Omega ) $
\[ \Gamma( L^{1} ) \mbox{-} \liminf_{ \eps \to 0 } \F_{ \eps }( u ) \ge \int_{ \Omega } f_{\hom}( \E u(x) ) \x. \]
Together with the proof of the $\limsup$-inequality from Subsection~\ref{section:Hencky-limsup} this yields the assertions in Remark~\ref{remark:Hencky-epilog}.
\end{remark}
\subsubsection{Singular points}
\label{subsection:singular-points}
\begin{trivlist}
\item
In order to control the behaviour in the singular points, 
we will have to assume that $f$ be asymptotically convex, as defined in \eqref{eq:asymptotic-convexity}. We first note that these $ c^{ \eta } $ can be assumed to enjoy the following additional properties. 
\begin{lemma}\label{lemma:c-properties} 
Suppose $f$ is as in Theorem~\ref{theo:Hencky-epilog}.
To every $ \eta > 0 $ there exist $ \beta_{\eta} > 0 $ and a Carath\'{e}odory function 
$ c^{\eta} : \R^{n} \times \M{n}{n}_{\sm} \to \R $ that is $ \Y^{n} $-periodic in the first variable and convex in the second 
such that \eqref{eq:asymptotic-convexity} is satisfied for a.e.~$ x \in \R^{n} $ and all $ X \in \M{n}{n}_{\sm} $. 
Moreover, $ c^{\eta} $ is non-negative with $ c^{\eta}(x,0) = 0 $ and $\dom ( c^{\eta} )^*(x, \cdot) $ is closed for a.e.\ $ x \in \R^{n} $. 
\end{lemma} 
\item 
Here $ (c^{\eta})^* $ denotes the convex conjugate of $ c^{\eta} $ in the second variable.  
\begin{proof}
Let us fix $ \eta > 0 $. For simplicity reasons we omit writing it as a superscript of the corresponding convex functions. 
By assumption there is a Carath\'{e}odory function $ \tilde{c} : \R^{n} \times \M{n}{n}_{\sm} \to \R $ 
that is $ \Y^{n} $-periodic in the variable and convex in the second verifying \eqref{eq:asymptotic-convexity} for some $ \tilde{ \beta }_{\eta} > 0 $. 
Setting $ \hat{c}(x,X) := \max \{ \tilde{c}(x,X) - \beta - \tilde{ \beta }_{1} , \alpha ( | X_{\dev} | + ( \tr X )^{2} ) \} $, 
we see that $ \hat{c} $ has a Hencky plasticity growth (with the coefficients $ \hat{\alpha} := \alpha $ and $ \hat{\beta} := \beta + 1 + \beta_{1} $) 
and satisfies \eqref{eq:asymptotic-convexity} with $ \beta_{\eta} := \tilde{ \beta }_{\eta} + \beta + \tilde{ \beta }_{1} $. 
Clearly, $ \hat{c} \ge 0 $ and $ \hat{c}(x,0) = 0 $ for every $ x $. 
\item
Elementary arguments show that the Hencky plasticity growth assumptions imply that for every $ x $ there is a closed set $K(x)$ in $ \M{n}{n}_{\dev} $ with 
\[ \{ Y \in \M{n}{n}_{\dev} : | Y | \le \hat{\alpha} \} 
   \subset K(x) 
   \subset \{ Y \in \M{n}{n}_{\dev} : | Y | \le \hat{\beta} \} \] 
such that the domain of the convex conjugate (with respect to the second variable) of $\hat{c}(x,\cdot)$ satisfies 
\[ \Int \dom \hat{c}^{*}(x,\cdot) = {\Int}_{\dev} K(x) + \R I
\quad \mbox{and} \quad
\overline{ \dom \hat{c}^{*}(x,\cdot) } = K(x) + \R I \]
for a.e.\ $x$. (Here $ \Int_{\dev} $ denotes the relative interior in $ \M{n}{n}_{\dev} $. 
\item
Since $ 0 \in \Int_{\dev} K(x) $, for $ \eps < 1 $ the set
\[ K_{ \eps }(x) := ( 1 - \eps ) K(x) \]
is a compact and convex subset of $ K(x) $ with $ 0 \in \Int_{\dev} K_{\eps}(x) $.
We choose a Carath\'{e}odory function $ c : \R^{n} \times \M{n}{n}_{\sm} \to \R $ such that 
\[ c(x,X) := \sup \{ X \cdot Y - \hat{c}^{*}(x,Y) : Y_{ \dev } \in K_{\eps}(x) \} \]
for a.e.\ $ x $ and all $ X $. In other words, $ c(x,X) = ( \hat{c}^{*}(x, \cdot) + \chi_{ K_{\eps}(x) + \R I } )^{*}(X) $ for a.e.\ $ x $ and all $ X $. (Note that since $ K_{\eps}(x) + \R I $ is closed and convex for every $ x $, $\chi_{ K_{\eps}(\cdot) + \R I } $ is a normal integrand.) As furthermore $c(x, \cdot)$ is convex, we have 
\[ c^{*}(x, \cdot) = ( \hat{c}^{*} + \chi_{ K_{\eps} + \R I } )^{**}(x, \cdot) = \hat{c}^{*}(x, \cdot) + \chi_{ K_{\eps}(x) + \R I } \]
for a.e.\ $ x $. From $ c^{*}(x, \cdot) \ge \hat{c}^{*}(x, \cdot) $, it follows
\[ c(x, \cdot) \le \hat{c}(x, \cdot) \]
for a.e.\ $ x $. 

On the other hand, as $ \hat{c} \ge 0 $, we have $ (1 - \eps)^{-1} \hat{c} \ge \hat{c} $, and consequently
\begin{align*}
  \hat{c}^{*}(x,Y) 
  &\ge ( 1 - \eps )^{-1} \hat{c}^{*}(x, (1 - \eps)Y) \\ 
  &= ( 1 - \eps )^{-1} \hat{c}^{*}(x, (1 - \eps)Y) + ( 1 - \eps )^{-1} \chi_{ K_{\eps}(x) }( (1 - \eps) Y_{\dev} ) 
   = \big( ( 1 - \eps )^{-1} c \big)^{*} (x, Y) 
\end{align*}
for a.e.\ $ x $ and all $ Y $, since $(1 - \eps) Y_{\dev} \notin K_{\eps}(x)$ implies $Y_{\dev} \notin K(x)$, i.e., $\hat{c}^{*}(x,Y) = + \infty$. Hence, we also have 
\[ \hat{c}(x, \cdot) \le (1 - \eps)^{-1} c(x, \cdot) \]
for a.e.\ $ x $. 
Summarizing, we have found a Carath\'{e}odory function $ c : \R^{n} \times \M{n}{n}_{\sm} \to \R $ which is $ \Y^{n} $-periodic in the variable and convex in the second such that $ \dom c^{*}(x,\cdot) = K_{\eps}(x) + \R I $ is closed and 
\[ c(x, \cdot) \le \hat{c}(x, \cdot) \le (1 - \eps)^{-1} c(x, \cdot) \]
holds and for a.e.\ $ x $. Taking $\eps$ sufficiently small, this estimate also shows that $c$ has Hencky plasticity growth and satisfies \eqref{eq:asymptotic-convexity} for $ 2 \eta $. 
\end{proof}
\item 
\begin{lemma}\label{lemma:dom-chom} 
For some $ \eta > 0 $, let $ c $ satisfy the assertions of Lemma~\ref{lemma:c-properties}. Then also $ c^*_{\hom} $ has a closed domain. 
\end{lemma} 
\begin{proof} 
We first note that 
\begin{align}\label{eq:chom-Darst} 
( c_{\hom} )^{*}( Y ) 
= \inf_{ { \Phi \in L^{2}( \Omega ; \M{n}{n}_{\sm} ) \atop \DIV \Phi = 0 } \atop \int_{ \Y^{n} } \Phi(x) dx = 0 } \int_{ \Y^{n} } c^{*}( x , Y + \Phi(x) ) \x, 
\end{align}
cf.\ \cite[Equation~(3.8)]{DemengelQi}. (This formula can be obtained by writing 
\begin{eqnarray*}
( c_{\hom} )^{*}( Y )
&  =  & \sup_{ X \in \M{n}{n}_{\sm} } 
	\left( Y \cdot X - \inf_{ \varphi \in LU_{0}( \Y^{n} ) } \int_{ \Y^{n} } c( x , X + \E \varphi(x) ) \x \right) \\
&  =  & \sup_{ X \in \M{n}{n}_{\sm} } 
	\left( - \inf_{ \varphi \in LU_{0}( \Y^{n} ) } \int_{ \Y^{n} } \Big( c( x , X + \E \varphi(x) ) - Y \cdot ( X + \E \varphi(x) ) \Big) \x \right) \\
&  =  & - \inf_{ u \in LU( \Y^{n} ) } F(u) + G( \Lambda u ), 
\end{eqnarray*}
with 
\[ F(u) := \chi_{C}(u) 
\quad \mbox{and} \quad
G( A ) := \int_{ \Y^{n} } \big( c( x , A(x) ) - Y \cdot A(x) \big) \x, \]
where $C := \{ u \in LU( \Y^{n} ) : u \mbox{ affine on } \partial \Y^{n} \}$ and 
\[ \Lambda u := \E u : LU( \Y^{n} ) \to \{ A \in L^{1}( \Y^{n} ; \M{n}{n}_{\sm} ) : \tr A \in L^{2} \}. \]
By duality, see, e.g., \cite[Remark III.4.2, Theorem III.4.1]{EkelandTemam}, we get 
\begin{eqnarray*} 
( c_{\hom} )^{*}( Y ) 
& = & \inf_{ \Phi \in L^{2}( \Y^{n} ; \M{n}{n}_{\sm}) \atop \Phi_{\dev} \in L^{\i} \}  } \big( F^{*}( \Lambda' \Phi ) + G^{*}( - \Phi ) \big), 
\end{eqnarray*}
where $\Lambda'$ denotes the adjoint of $\Lambda$. 
From this, straightforward calculations lead to \eqref{eq:chom-Darst} with the help of a measurable selection argument 
(see, e.g., \cite[Proposition IV.1.2]{EkelandTemam}) and the observation that $ \dom c^{*}( x , \cdot ) \cap \M{n}{n}_{\dev} $ is bounded uniformly in $ x $.)

As a direct consequence we have $Y \in \dom ( c_{\hom} )^{*}$ if and only if  
\begin{multline*}
\exists \Phi \in L^{2}( \Y^{n} ; \M{n}{n}_{\sm} ) : \DIV \Phi = 0, \ \int_{ \Y^{n} } \Phi(x) \x = 0, \ ( Y + \Phi(x) )_{\dev} \in K(x) \mbox{ for a.e.~} x \in \Y^{n}.  
\end{multline*}
In order to show that $ \dom ( c_{\hom} )^{*} $ is closed, we suppose $ \{ Y^{(j)} \}_{ j \in \N } \subset \dom ( c_{\hom} )^{*} $
and $ Y^{(j)} \to Y $. Denote $ M := \sup_{j} | Y^{(j)} | < \i $.

\item
Let $ \Phi^{(j)} $ be a corresponding function to $ Y^{(j)} $. Then for almost every $ x \in \Y^{n} $
\[ \Phi_{\dev}^{(j)}(x) \in K(x) - Y_{\dev}^{(j)} \subset \{ X \in \M{n}{n}_{\dev} : |X| \le \beta + M \}. \]
Hence, $ \{ \Phi^{(j)}_{\dev} \}_{ j \in \N } $ is a bounded sequence in $ L^{\i}( \Y^{n} ; \M{n}{n}_{\sm} ). $
Moreover, 
\[ 0
= \DIV \Phi^{(j)}
= \DIV( \Phi^{(j)}_{ \dev } + \tfrac{ \tr \Phi^{(j)} }{n} I ) 
= \DIV \Phi^{(j)}_{ \dev } + \DIV ( \tfrac{ \tr \Phi^{(j)} }{n} I )
= \DIV \Phi^{(j)}_{ \dev } + \nabla ( \tfrac{ \tr \Phi^{(j)} }{n} ). \]
Hence, in the sense of distributions in $ W^{-1,2}( \Y^{n} ; \M{n}{n}_{\sm} ) $
\[ \nabla ( \tr \Phi^{(j)} ) = - n \DIV \Phi^{(j)}_{ \dev }, \]
and thus
\[ \| \nabla ( \tr \Phi^{(j)} ) \|_{ W^{-1,2} }
= n \| \DIV \Phi^{(j)}_{ \dev } \|_{ W^{-1,2} }
\le n \| \Phi^{(j)} \|_{ L^{\i} }
\le n ( \beta + M ). \]
Since $ \int_{ \Y^{n} } \tr \Phi^{(j)}(x) \, dx = 0 $, a weak form of the Poincar{\'e} inequality (see, e.g., \cite[p.~175]{Galdi}) leads to 
\[ \| \tr \Phi^{(j)} \|_{ L^{2} } \le C \| \nabla ( \tr \Phi^{(j)} ) \|_{ W^{-1,2} } \le C n ( \beta + M ). \]
Thus we have proved that $ \{ \Phi^{(j)} \}_{ j \in \N } $ is a bounded sequence in $ L^{2}( \Y^{n} ; \M{n}{n}_{\sm} ) $.
Hence, it contains a non-relabeled weakly converging subsequence, say
\[ \Phi^{(j)} \weakly \Phi 
\quad \mbox{in } L^{2}. \]
We wish to prove that $ \Phi $ an appropriate function for $ Y $, i.e. 
\[ \DIV \Phi = 0, \ \int_{ \Y^{n} } \Phi(x) \x = 0
\quad \mbox{and} \quad
( Y + \Phi(x) )_{\dev} \in K(x) \mbox{ for a.e.~} x \in \Y^{n}. \]
The first two properties follow immediately from the weak convergence in $ L^{2}( \Y^{n} ; \M{n}{n}_{\sm} ) $. As for the last,
let us define the set
\[ \Xi := \{ \Psi \in L^{2}( \Y^{n} ; \M{n}{n}_{\sm} ) : \Psi_{\dev}(x) \in K(x) \mbox{ for a.e. } x \in \Y^{n} \}. \]
This set is convex and closed (in the norm topology), and is therefore also weakly closed.
Since 
\[ \Xi \owns Y^{(j)} + \Phi^{(j)} \weakly Y + \Phi 
\quad \mbox{in } L^{2}, \]
also $ Y + \Phi \in \Xi $.
\end{proof}
\item
Let $c$ have all the properties stated in Lemma~\ref{lemma:c-properties}. We introduce 
\[ {\cal C}_{\eps}(u) := \left\{ \begin{array}{ll}
\int_{ \Omega } c \big( \tfrac{x}{ \eps } , \E u(x) \big) \x, & u \in LU( \Omega ), \\
\i, & {\rm else.}
\end{array} \right. \]
In \cite{DemengelQi} the authors introduce for every non-negative function $ \varphi \in C( \overline{ \Omega } ) $ also the functionals
\[ \langle {\cal C}_{\eps}(u) , \varphi \rangle := \left\{ \begin{array}{ll}
\int_{ \Omega } c \big( \tfrac{x}{ \eps } , \E u(x) \big) \ \varphi(x) \x, & u \in LU( \Omega ), \\
\i, & {\rm else.}
\end{array} \right. \]
In Theorem 1.1 they show that for any $ \eps_{j} \searrow 0 $ there exists a subsequence $ \{ j_{k} \}_{k \in \N} $ such that
$ \Gamma( L^{q} ) $-$ \lim_{k \to \i} \langle {\cal C}_{ \eps_{j_{k}} }( u ) , \varphi \rangle $ exists 
for every non-negative continuous $ \varphi : \overline{ \Omega } \to \R $ and every $ u \in U( \Omega ) $. 
($q$ is as before, i.e. $ 1 < q < \frac{n}{n-1} $.)
In Proposition 2.1 it is proved that for $ u \in LU( \Omega ) $ the corresponding $ \Gamma $-limit is given by a density,
which is by Proposition 2.2 location-independent.
\item
If $ \varphi = 1 $, we get the existence of 
$ \Gamma( L^{1} ) $-$ \lim_{k \to \i} {\cal C}_{ \eps_{j_{k}} } $
on the whole $ L^{1}( \Omega ; \R^{n} ) $ with the domain $ U( \Omega ) $.
The passage from the $ \Gamma( L^{q} ) $-limit to the $ \Gamma( L^{1} ) $-limit
follows from the lower bound and the compactness of the embedding $ U( \Omega ) \hookrightarrow L^{q}( \Omega ; \R^{n} ) $. 
By Remark~\ref{remark:Hencky-epilog} the density of the $ \Gamma $-limit must be $ c_{ \hom } $.
\item
Morever, they show that this density determines the $ \Gamma $-limit for every $ u \in U( \Omega ) $ with the formula
\[ \int_{ \Omega } \varphi(x) \ d \big( c_{\hom}( E u ) \big)(x). \]
Under the integral there is a measure $ c_{\hom}( E u ) $ that still needs to be explained.
Before that, let us notice that the expression neither depends 
on $ \{ \eps_{j} \}_{ j \in \N } $ nor on $ \{ j_{k} \}_{ k \in \N } $.
By the Urysohn property it follows 
that actually even 
$ \Gamma( L^{q} ) $-$ \lim_{ \eps \to 0 } \langle {\cal C}_{ \eps }( \prostor ) , \varphi \rangle $ 
and therefore $ \Gamma( L^{1} ) $-$ \lim_{ \eps \to 0 } {\cal C}_{ \eps } $ exist and are given by $ c_{\hom} $.
\item
Now we return to the definition of $ c_{\hom}( E u ) $. 
For convex functions with a possible superlinear growth, this was done in \cite{DemengelTemam}.
However, there are some requirements that have to be met (see Subsection 2.2 therein). 
Right away we see that $ c_{\hom} $ is a non-negative finite convex function with $ c_{\hom}(0) = 0 $.
Since $ c_{\hom} $ has superlinear growth, its asymptotic function $ ( c_{\hom} )^{\#} $ does not coincide with the recession function. 
The latter is in this case for all $ X \in \M{n}{n}_{\sm} $
\[ ( c_{\hom} )^{\i}(X) = \limsup_{t \to \i, \ Y \to X} \frac{ c_{\hom}( t Y ) }{t} = \i. \]
However, by Remark~\ref{remark:c-strict-continuity} we still have 
\[ ( c_{\hom} )^{\#}|_{ \M{n}{n}_{\dev} } = ( c_{\hom}|_{ \M{n}{n}_{\dev} } )^{\i}. \]
This distinction is actually a very important issue in this analysis.
To emphasize the difference, our denotation differs from the one in \cite{DemengelTemam}.
The domain of $ ( c_{\hom} )^{\#} $ is $ \M{n}{n}_{\dev} $, 
and for $ X \in \M{n}{n}_{\dev} $ it holds $ c_{\hom}(X) \le \beta ( |X| + 1 ) $. Finally, by Lemma~\ref{lemma:dom-chom} the domain of $ ( c_{\hom} )^{*} $ is closed. 
\item
Then, according to \cite[Section 2.2]{DemengelTemam}, we may define
\[ c_{\hom}( E u ) := c_{ \hom }( \E u ) \ \L^{n} + ( c_{\hom} )^{\#}( \tfrac{ d E^{s} u }{ d | E^{s} u | } ) \ | E^{s} u |. \]
Therefore, for every non-negative $ \varphi \in C( \overline{ \Omega } ) $ and $ u \in U( \Omega ) $,
\[ \Gamma( L^{q} ) \mbox{-} \lim_{ \eps \to \i } \langle { \cal C }_{\eps}( u ) , \varphi \rangle
= \langle { \cal C }_{\hom} ( u ) , \varphi \rangle 
:= \int_{ \Omega } \varphi(x) \ d c_{\hom}( E u )(x). \]
%
%
%
As before this implies
\[ \Gamma( L^{1} ) \mbox{-} \lim_{ \eps \to \i } { \cal C }_{\eps}
= { \cal C }_{\hom} \]
where 
\[ {\cal C}_{\hom}(u) 
:= \left\{ \begin{array}{ll}
 \int_{ \Omega } c_{\hom}( E u ), 
& u \in U( \Omega ), \\
\i, 
& {\rm else.}
\end{array} \right. 
\]
\item
To simplify the denotation, let us define the asymptotic function also for non-convex functions as
\[ f^{\#}(X) := \limsup_{ t \to \i } \frac{ f(tX) }{t}. \]
As already stated in Remark~\ref{remark:c-strict-continuity}, if $f$ is Lipschitz continuous, it coincides with the recession function. 
\begin{lemma}
\label{lemma:liminf-singular}
Suppose $ f : \R^{n} \times \M{n}{n}_{\sm} \to \R $ is a Carath\'{e}odory function 
that is $ \Y^{n} $-periodic in the first variable and satisfies \eqref{eq:Hencky-growth} and \eqref{eq:asymptotic-convexity}. 
Then 
\[ \mu^{s} \ge ( f_{\hom} )^{\#}( \tfrac{ d E^{s} u }{ d | E^{s} u | } ) | E^{s} u |. \]
\end{lemma}
\begin{proof}
Take any non-negative $ \varphi \in C_{0}( \Omega ) $ and, for given $\eta > 0$, 
let $ c^{\eta} $ be as provided by Lemma~\ref{lemma:c-properties}. 
We may add to our assumptions on $ \{ u_{j} \}_{ j \in \N } $ from the beginning of Subsection~\ref{section:Hencky-liminf} that $ ( | \E_{\dev} u_{j} | + ( \DIV u_{j} )^{2} ) \L^{n} \weakstar \sigma $ in $ M( \Omega ) $.
Because of the weak-$*$ convergence and from 
\[ f(x,X) \ge c^{ \eta }(x,X) - \eta ( | X_{\dev} | + ( \tr X )^{2} ) - \beta_{\eta}, \] 
it follows (denoting the functionals corresponding to $c^{ \eta } $ and $c^{ \eta }_{\hom} $ by ${ \cal C }^{ \eta }_{ \eps } $, respectively, $ { \cal C }^{ \eta }_{ \hom } $)
\begin{align*}
& \int_{ \Omega } \varphi(x) \ d \mu(x) \\
&  =  \lim_{ j \to \i } \int_{ \Omega } \varphi(x) \ d \mu_{j}(x) \\
& \ge \liminf_{ j \to \i } \langle { \cal C }^{ \eta }_{ \eps_{j} }( u_{j} ) , \varphi \rangle 
	- \lim_{j \to \i} \eta \int_{ \Omega } ( | \E_{\dev} u_{j}(x) | + ( \DIV u_{j} )^{2} ) \varphi(x) \x - \beta_{\eta} \int_{ \Omega } \varphi(x) \x \\
& \ge \langle { \cal C }^{ \eta }_{ \hom }( u ) , \varphi \rangle - \eta \, \langle \sigma , \varphi \rangle - \beta_{\eta} \int_{ \Omega } \varphi(x) \x.
\end{align*}
Hence,
$ \mu \ge c^{ \eta }_{\hom}( E u ) - \eta \, \sigma - \beta_{\eta} \L^{n} $. The inequality holds also for the corresponding singular part, i.e.,
\[ \mu^{s} \ge ( c^{ \eta }_{\hom} )^{\#}( \tfrac{ d E^{s} u }{ d | E^{s} u | } ) | E^{s} u | - \eta \, \sigma^{s}. \]
From
\[ | X_{\dev} | + ( \tr X )^{2} \le \frac{ f(x,X) }{ \alpha }, \]
it follows
\[ \Big( 1 - \frac{ \eta }{ \alpha } \Big) \ f(x,X) - \beta_{\eta} 
\le c^{ \eta }(x,X)
\le \Big( 1 + \frac{ \eta }{ \alpha } \Big) \ f(x,X) + \beta_{\eta}. \]
Hence,
\[ \Big( 1 - \frac{ \eta }{ \alpha } \Big) \ f_{\hom}(X) - \beta_{\eta} 
\le c^{ \eta }_{\hom}(X)
\le \Big( 1 + \frac{ \eta }{ \alpha } \Big) \ f_{\hom}(X) + \beta_{\eta}, \]
and thus 
\[ \Big( 1 - \frac{ \eta }{ \alpha } \Big) ( f_{\hom} )^{\#}(X)
\le ( c^{ \eta }_{\hom} )^{\#}(X) 
\le \Big( 1 + \frac{ \eta }{ \alpha } \Big) ( f_{\hom} )^{\#}(X). \]
This holds for every $ \eta > 0 $, so 
\[ ( f_{\hom} )^{\#}(X) = \lim_{ \eta \to 0 }( c^{ \eta }_{\hom} )^{\#}(X), \]
and, since $ E^{s} u $ is supported on deviatoric matrices, by dominated convergence 
\[ \mu^{s} \ge ( f_{\hom} )^{\#} \big( \tfrac{ d E^{s} u }{ d | E^{s} u | } \big) | E^{s} u |. \qedhere \] 
\end{proof}
\end{trivlist}

\subsection{Relaxation at zero hardening revisited}
\label{section:Hencky-relax-KK}
\begin{trivlist}
\item
The considerations of the previous subsections may be applied to the functionals with location-independent densities.
Clearly, thus we are investigating the relaxation of the homogeneous setting with or without hardening.
\item
We would like to reconsider the assumption of asymptotic convexity for this case
taking into account the recent progress in this field, i.e.~the results in \cite{KirchheimKristensen}.
For the sake of simplicity, we start with $f$ that is already symmetric-quasiconvex.
\item
Suppose that $f$ for some $ 0 \le \gamma < 1 $ fulfils
\[ f(X) \ge f^{\#}( X_{ \dev } ) - \beta ( | X_{ \dev } |^{\gamma} + 1 ) \]
for all $ X \in \M{n}{n}_{\sm} $. Here we actually have two assumptions in mind. 
\item
First, we make the projection
\[ f(X) \ge f( X_{ \dev } ) - C_{1} ( | X_{ \dev } |^{\gamma} + 1 ). \]
Recalling the quadratic growth in the trace direction, it is not a very strong assumption. 
However, it cannot be excluded just by using the lower bound and Lemma~\ref{lemma:Lipschitz-Hencky}.
Currently, it is not known to us whether alone the symmetric-quasiconvexity with a Hencky plasticity growth condition suffices for this estimate.
\item
Additionally, we suppose for $ P \in \M{n}{n}_{\dev} $ also
\[ f(P) \ge f^{\#}(P) - C_{2}( | P |^{\gamma} + 1 ). \]
Such a behaviour is in accordance with similar assumptions in the literature (e.g., \cite{BarrosoFonsecaToader,BouchitteFonsecaMascarenhas}). 
\item
Thus, we imposed a lower bound on the function $f$,
however, not by a convex function as before, but by a 1-homogeneous symmetric-rank-one-convex function (on $ \M{n}{n}_{\dev} $).
Let us be more precise:
\begin{itemize}
\item
$ V := \M{n}{n}_{\dev} $ is a finite-dimensional normed space,
\item 
$ D := \{ a \odot b : a, b \in \R^{n}, \ a \perp b \} $ spans $V$,
\item
$ g := f^{\#}|_{ \M{n}{n}_{\dev} } : V \to \R $ is 1-homogeneous and is convex along any direction from $D$.
\end{itemize}
According to \cite[Theorem 1.1]{KirchheimKristensen}, $g$ is convex at every point from $D$.
To be more specific, for every pair of orthogonal vectors $ a, b \in \R^{n} $, there exists a (homogeneous) linear function $ \ell : V \to \R $
such that
\[ g \ge \ell \mbox{ everywhere on $V$}
\quad \mbox{and} \quad
g( a \odot b ) = \ell( a \odot b ). \] 
With this tools we may prove an alternative relaxation result. 
\begin{prop}
\label{lemma:Hencky-relax-2}
Let us have a symmetric-quasiconvex function $ f : \M{n}{n}_{\sm} \to \R $ with Hencky plasticity growth \eqref{eq:Hencky-growth} for which
there exists $ \gamma \in [0,1) $ such that
\[ f(X) \ge f^{\#}( X_{ \dev } ) - \beta ( | X_{ \dev } |^{\gamma} + 1 ) \]
for all $ X \in \M{n}{n}_{\sm} $.
Then the lower semicontinuous envelope of the functional 
\[ \F(u) := \left\{ \begin{array}{ll}
\int_{ \Omega } f \big( \E u(x) \big) \x, & u \in LU( \Omega ; \R^{n} ), \\
\i, & {\rm else,}
\end{array} \right. \]
is 
\[ \lsc \F(u)
= \left\{ \begin{array}{ll}
\int_{ \Omega } f \big( \E u(x) \big) \x 
+ \int_{ \Omega } f^{\#} \big( \tfrac{ d E^{s} u }{ d | E^{s} u | }(x) \big) \ d| E^{s} u |(x), 
& u \in U( \Omega ; \R^{n} ), \\
\i, & {\rm else.}
\end{array} \right. \]
\end{prop}
\begin{proof}
By Remark~\ref{remark:Hencky-epilog} and Lemma~\ref{lemma:liminf-regular}, 
we just have to prove the $ \liminf $-inequality in the singular points.
\item
Let us take any $ u \in U( \Omega ; \R^{n} ) $ and let $ u_{j} \to u $ in $ L^{1}( \Omega ; \R^{n} ) $.
Clearly, we may for $ \{ u_{j} \}_{ j \in \N} $ consider only bounded sequences in $ LU( \Omega ; \R^{n} ) $.
Moreover, we may suppose
\begin{itemize}
\item
$ u_{j} \weakly u $ in $ U( \Omega ; \R^{n} ) $,
\item
$ \mu_{j} := f( \E u_{j} ) \L^{n} $ converge weakly-$*$ to some $ \mu $ in $ M( \Omega ; \R^{n} ) $,
\item
$ | \E_{\dev} u_{j} |^{\gamma} + 1 $ converge weakly to some $h$ in $ L^{1/\gamma}( \Omega ; \R^{n} ) $.
\end{itemize}
Our goal is to show 
$ \mu^{s} \ge f^{\#}( \frac{ d E^{s} u }{ d | E^{s} u | } ) | E^{s} u | $ with
$ \mu = g \, \L^{n} + \mu^{s} $ again being the Lebesgue decomposition of $ \mu $ with respect to $ \L^{n} $.
\item
According to Theorem~\ref{theo:BD-Alberti}, 
for $ | E^{s} u | $-a.e.~$ x_{0} \in \Omega $, there exist $ a( x_{0} ) , b( x_{0} ) \in \R^{n} $ such that
\begin{equation}
\label{eq:Hencky-revisited-1}
\frac{ d E^{s} u }{ d | E^{s} u | }( x_{0} ) 
= \lim_{ \rho \to 0 } \frac{ Eu( B_{\rho}( x_{0} ) ) }{ |Eu|( B_{\rho}( x_{0} ) ) }
= a( x_{0} ) \odot b( x_{0} ). 
\end{equation}
Since $ \tr E^{s} u = 0 $, we have $ \tr a( x_{0} ) \odot b( x_{0} ) = a( x_{0} ) \cdot b( x_{0} ) = 0 $.
By the Besicovitch derivation theorem~\ref{theo:Besicovitch},
for $ | E^{s} u | $-a.e.~$ x_{0} \in \Omega $ also,
\begin{equation}
\label{eq:Hencky-revisited-2}
\frac{ d \mu^{s} }{ d | E^{s} u | }( x_{0} ) 
= \lim_{ \rho \to 0 } \frac{ \mu( B_{\rho}( x_{0} ) ) }{ | E u |( B_{\rho}( x_{0} ) ) }.
\end{equation}
Since $ ( \DIV u ) \ \L^{n} $ and $ h \, \L^{n} $ are each mutually singular with $ | E^{s} u | $,
for $ | E^{s} u | $-a.e.~$ x_{0} \in \Omega $ also
\begin{equation}
\label{eq:Hencky-revisited-3}
\lim_{ \rho \to 0 } \frac{ \int_{ B_{\rho}( x_{0} ) } \DIV u(x) \x }{ | E u |( B_{\rho}( x_{0} ) ) } 
= \lim_{ \rho \to 0 } \frac{ \int_{ B_{\rho}( x_{0} ) } h(x) \x }{ | E u |( B_{\rho}( x_{0} ) ) } 
= 0. 
\end{equation} 
\item
Let $ x_{0} \in \Omega $ be from now on any point 
where \eqref{eq:Hencky-revisited-1}, \eqref{eq:Hencky-revisited-2} and \eqref{eq:Hencky-revisited-3} hold. 
Being fixed, we stop writing $ x_{0} $ in the denotations.
We need to show that
\[ \lim_{ \rho \to 0 } \frac{ \mu( B_{\rho} ) }{ |Eu|( B_{\rho} ) } \ge f^{\#}( a \odot b ). \]
Let $ \ell : \M{n}{n}_{\dev} \to \R $ be a linear function from \cite[Theorem~1.1]{KirchheimKristensen}
that determines the supporting hyperplane for $ f^{\#}|_{ \M{n}{n}_{\dev} } $ at $ a \odot b $. 
For all but countable many $ \rho > 0 $ it holds 
\begin{eqnarray*}
\mu( B_{\rho} )
&  =  & \lim_{ j \to \i } \mu_{j}( B_{\rho} ) \\
&  =  & \lim_{ j \to \i } \int_{ B_{\rho} } f( \E u_{j}(x) ) \x  \\
& \ge & \limsup_{ j \to \i } \int_{ B_{\rho} } f^{\#}( \E_{\dev} u_{j}(x) ) \x 
	- \beta \lim_{ j \to \i } \int_{ B_{\rho} } \big( | \E_{\dev} u_{j}(x) |^{\gamma} + 1 \big) \x  \\
& \ge & \limsup_{ j \to \i } \int_{ B_{\rho} } \ell( \E_{\dev} u_{j}(x) ) \x 
	- \beta \int_{ B_{\rho} } h(x) \x \\
&  =  & \ell \big( \dev E u( B_{\rho} ) \big)
	- \beta \int_{ B_{\rho} } h(x) \x. 
\end{eqnarray*}
Hence, by \eqref{eq:Hencky-revisited-3} and \eqref{eq:Hencky-revisited-1}
\[ \lim_{ \rho \to 0 } \frac{ \mu( B_{\rho} ) }{ | E u |( B_{\rho} ) } 
\ge \limsup_{ \rho \to 0 } \ell \left( \frac{ \dev E u( B_{\rho} ) }{ | E u |( B_{\rho} ) } \right) 
= \ell \left( \lim_{ \rho \to 0 } \frac{ E u( B_{\rho} ) }{ | E u |( B_{\rho} ) } \right) 
= \ell( a \odot b ). \]
Now, we employ
\[ \ell( a \odot b )
= f^{\#}( a \odot b ). \qedhere \]
\end{proof}
\end{trivlist}
\begin{appendix}
\section{Appendix: Miscellaneous auxiliary results}
\begin{trivlist}
\item 
For convenience of the reader we review the notion of quasiconvexity on linear strains and collect a couple of auxiliary results in the specific form they were applied above. 
\item
\begin{definition}\label{defi:sym-rk-one}
A locally bounded Borel function $ f : \M{n}{n}_{\sm} \to \R $ is {\it symmetric-quasi\-convex} (resp.~{\it symmetric-rank-one convex})
if the function
\[ \M{n}{n} \to \R, \quad X \mapsto f( X_{\sm} ) \]
is quasiconvex (resp.~rank-one convex).
\end{definition}
\item
Therefore, a symmetric-quasiconvex function $f$ must fulfil 
\[ \int_{ \Y^{n} } f( X + \E \varphi(x) ) \x \ge f(X) \]
for every $ X \in \M{n}{n}_{\sm} $ and every $ \varphi \in C^{\i}_{c}( \Y^{n} ; \R^{n} ) $ whereas
symmetric-rank-one convexity means that
\[ t \mapsto f( X + t \ a \odot b ) \]
is convex for all $ X \in \M{n}{n}_{\sm} $ and $ a , b \in \R^{n} $. 
We denote the {\it symmetric-quasiconvex envelope} by $ f^{\qcls} $. It is related to the quasiconvex envelope by the formula
\[ (f \circ \sm)^{\qc} = f^{\qcls} \circ \sm. \]
($\sm : \M{n}{n} \to \M{n}{n}_{\sm}$ is simply the symmetrizing projection.)
For the proof and other properties, we refer to \cite{Zhang:04}.
\item
The following estimate on the Lipschitz constant is proved in \cite[Lemma~2.2]{BKK}:
\begin{lemma}
\label{lemma:BKK}
If $ f : B_{2r}( X_{0} ) \subset \M{m}{n} \to \R $ is separately convex, then
\[ {\rm lip}( f ; B_{r}( X_{0} ) ) \le \sqrt{mn} \frac{ {\rm osc}( f ; B_{2r}( X_{0} ) ) }{r}, \]
where ${\rm osc}( f ; U ) = \sup \{ | f(X) - f(Y) | : X,Y \in U \}$. 
\end{lemma}
\item 
Next we state an equiintegrability result of Fonseca, M{\"u}ller and Pedregal,
see \cite[Lemma~1.2]{FonsecaMuellerPedregal} and \cite[Lemma~8.3]{Pedregal}:
\begin{lemma}
\label{lemma:equi-int}
Let $ \Omega \subset \R^{n} $ be an open bounded set, 
and let $ \{ u_{i} \}_{ i \in \N } $ be a bounded sequence in $ W^{1, p}( \Omega ; \R^{m} ) $, $ 1 < p < \i $. 
There exist a subsequence $ \{ u_{ i_{k} } \}_{ k \in \N } $ 
and a sequence $ \{ v_{k} \}_{ k \in \N } \subset W^{1, p}(\Omega; \R^m) $ such that 
\[ \lim_{k \to \i} \big| \{ \nabla v_{k} \ne \nabla u_{ i_{k} } \} \cup \{ v_{k} \ne u_{ i_{k} } \} \big| = 0 \] 
and $ \{ |\nabla v_k|^{p} \}_{ k \in \N } $ is equiintegrable. 
Moreover, if $ u_{i} \weakly u $ in $ W^{1, p}(\Omega; \R^{m} ) $, 
then the $ v_{k} $ can be chosen in such a way that $ v_{k} = u $ on $ \partial \Omega $ 
and $ v_{k} \weakly u $ in $ W^{1, p}(\Omega; \R^m) $. 
\end{lemma} 
\item 
The following version of the Besicovitch derivation theorem is shown in \cite[Theorem~1.153]{FonsecaLeoni}.
\begin{theo}
\label{theo:Besicovitch}
Let $ \mu, \nu $ be two positive regular Borel measures on $ \R^{n} $.
There exists a Borel set $ N \subset \R^{n} $ with $ \mu(N) = 0 $ 
such that for any $ x \in \R^{n} \setminus N $ and any convex compact neighbourhood of the origin $ C \subset \R^{n} $
\[ \frac{ d \nu^{a} }{ d \mu }(x) = \lim_{ r \searrow 0 } \frac{ \nu( x + r C ) }{ \mu( x + r C ) } \in \R \]
and
\[ \lim_{ r \searrow 0 } \frac{ \nu^{s}( x + r C ) }{ \mu( x + r C ) } = 0, \]
where
\[ \nu = \nu^{a} + \nu^{s},
\quad
\nu^{a} \ll \mu
\quad \mbox{and} \quad
\nu^{s} \perp \mu. \]
\end{theo}
\item
The following theorem gathers the relevant results from \cite[Section 2]{BorchersSohr}. 
See also the references therein, as the original proofs go back to Bogovskii.
Therefore, $ {\mathcal B} $ is sometimes referred to as Bogovskii's operator.
\item
\begin{theo}
\label{theo:Galdi}
Let $ \Omega \subset \R^{n} $ be a bounded Lipschitz domain and $ 1 < q < \i $.
There exists a linear operator $ {\mathcal B} = {\mathcal B}_{\Omega,q}  : L^{q}( \Omega ) \to W^{1,q}_{0}( \Omega ; \R^{n} ) $ 
with the following properties:
\begin{itemize}
\item 
For every $ f \in L^{q}( \Omega ) $ with $ \int_{\Omega} f(x) \x = 0 $, it holds
\[ \DIV {\mathcal B} f = f. \]
\item 
For every $ f \in L^{q}( \Omega ) $
\[ \| \nabla ( {\mathcal B} f ) \|_{ L^{q}( \Omega ; \M{n}{n} ) } \le C \| f \|_{ L^{q}( \Omega ) }. \]
The constant $C$ depends only on $ \Omega $ and $q$, and is translation- and scaling-invariant.
\item 
If $ f \in C^{\i}_{c}( \Omega ) $, then $ {\mathcal B} f \in C^{\i}_{c}( \Omega ; \R^{n} ) $.
\end{itemize}
\end{theo}
\end{trivlist}
\end{appendix}

\typeout{References}

\end{document}